\documentclass{amsart}
\usepackage{enumitem}
\usepackage{amssymb}
\usepackage{amsthm}
\usepackage{amsmath}
\usepackage{graphicx}
\usepackage[all]{xy}
\usepackage{colonequals}

\usepackage[usenames]{color}
\usepackage{xcolor}
\usepackage[colorlinks,citecolor=blue,linkcolor=red!80!black]{hyperref}
\usepackage{tikz}
\usepackage{ifthen}
\usetikzlibrary{angles,quotes}
\usetikzlibrary{fit, calc}
\usetikzlibrary{decorations.pathreplacing}
\usepackage[nameinlink]{cleveref}
\usepackage{todonotes}
\usepackage{verbatim}
\usepackage{soul} 

\theoremstyle{plain}
\newtheorem{proposition}{Proposition}[section]
\newtheorem{theorem}[proposition]{Theorem}
\crefname{theorem}{Theorem}{Theorems}
\Crefname{theorem}{Theorem}{Theorems}
\newtheorem*{theorem*}{Theorem}

\newtheorem{lemma}[proposition]{Lemma}
\newtheorem{corollary}[proposition]{Corollary}

\newtheorem{introthm}{Theorem}

\newtheorem{claim}[proposition]{Claim}

\theoremstyle{definition}
\newtheorem{example}[proposition]{Example}
\newtheorem{definition}[proposition]{Definition}

\theoremstyle{remark}
\newtheorem{remark}[proposition]{Remark}

\newtheorem{question}[proposition]{Question}

\numberwithin{equation}{section}

\DeclareMathOperator{\Mod}{\operatorname{Mod}}
\DeclareMathOperator{\Homeo}{\operatorname{Homeo}}

\DeclareMathOperator{\diam}{diam}

\DeclareMathOperator{\stab}{stab}
\def\co{\colon\thinspace}

\DeclareMathOperator{\Cc}{\mathcal{C}}

\DeclareMathOperator{\Gc}{\mathcal{G}}
\DeclareMathOperator{\Hc}{\mathcal{H}}

\newcommand{\Z}{\mathbb{Z}}
\newcommand{\bgit}{{\sf M}}
\newcommand{\bsi}{{\sf B}}
\newcommand{\abs}[1]{\left\vert#1\right\vert}

\newcommand{\vertiii}[1]{{\left\vert\kern-0.25ex\left\vert\kern-0.25ex\left\vert #1
		\right\vert\kern-0.25ex\right\vert\kern-0.25ex\right\vert}}
\newcommand{\<}{\langle}
\renewcommand{\>}{\rangle}
\newcommand{\arr}{\rightarrow}
\newcommand{\del}{\partial}
\newcommand{\inv}{^{-1}}
\newcommand{\define}[1]{\emph{#1}}
\newcommand{\cc}{\mathcal{C}} 

\usepackage{setspace}
\newcounter{notes}

\newcommand{\hbpar}[1] {\refstepcounter{notes}\todo[inline,color=cyan!30,linecolor=cyan!80!black,size=\normalsize    ]{(\arabic{notes}-hb) #1}}

\newcommand{\hhpar}[1] {\refstepcounter{notes}\todo[inline,color=yellow!30,linecolor=orange!80!black,size=\normalsize]{(\arabic{notes}-hh) #1}}

\usepackage{framed}
\definecolor{shadecolor}{RGB}{226, 223, 235}

\definecolor{lav}{RGB}{198, 173, 240}

\begin{document}

\title[Constructing reducibly geometrically finite subgroups]{Constructing reducibly geometrically finite subgroups of the mapping class group}

	\author[Aougab]{Tarik Aougab}
	\address{Department of Mathematics and Statistics, Haverford College}
	\email{taougab@haverford.edu}
	\urladdr{sites.google.com/view/tarikaougab/}

	\author[Bray]{Harrison Bray}
	\address{Department of Mathematical Sciences, George Mason University}
	\email{hbray@gmu.edu}
	\urladdr{www.harrisonbray.com}

	\author[Dowdall]{Spencer Dowdall}
	\address{Department of Mathematics, Vanderbilt University}
	\email{spencer.dowdall@vanderbilt.edu}
	\urladdr{https://math.vanderbilt.edu/dowdalsd/}

	\author[Hoganson]{Hannah Hoganson}
	\address{Department of Mathematics, University of Maryland}
	\email{hoganson@umd.edu}
	\urladdr{https://www.math.umd.edu/~hoganson/}

	\author[Maloni]{Sara Maloni}
	\address{Department of Mathematics, University of Virginia}
	\email{sm4cw@virginia.edu}
	\urladdr{sites.google.com/view/sara-maloni/}

	\author[Whitfield]{Brandis Whitfield}
	\address{Department of Mathematics, Temple University}
	\email{brandis@temple.edu}
	\urladdr{sites.google.com/view/algebrandis/}

	\date{\today}

	\begin{abstract}
		In this article, we consider qualified notions of geometric finiteness in mapping class groups called \emph{parabolically geometrically finite} (PGF) and \emph{reducibly geometrically finite} (RGF). We examine several constructions of subgroups and determine when they produce a PGF or RGF subgroup. These results provide a variety of new examples of PGF and RGF subgroups. Firstly, we consider the right-angled Artin subgroups constructed by Koberda \cite{Koberda-RAAG_subgroups_of_MCGs} and Clay--Leininger--Mangahas \cite{CLM-RAAGs_in_MCG}, which are generated by high powers of given elements of the mapping class group. We give conditions on the supports of these elements that imply the resulting right-angled Artin subgroup is RGF. Secondly, we prove combination theorems which provide conditions for when a collection of reducible subgroups, or sufficiently deep finite-index subgroups thereof, generate an RGF subgroup.
	\end{abstract}

	\maketitle

\setcounter{tocdepth}{1}
	\tableofcontents

	\section{Introduction}\label{intro}

Motivated by a long standing analogy with the classical theory of Kleinian groups, Farb and Mosher \cite{FarbMosher-ConvexCocompact} introduced the notion of a \define{convex cocompact} subgroup of the mapping class group $\Mod(S)$ of a surface $S$. These subgroups have received much attention since their introduction in 2002, and there is now a well-developed theory that connects them to hyperbolicity of surface group extensions, to the geometry of Teichm\"uller space, and to the geometry and distance formula of the mapping class group 
\cite{BestvinaBrombergKentLeininger-UndistortedPurelyPA,DurhamTaylor-StabilityInMCGs,KentLeininger-Shadows,Hamenstadt05}. In particular, there are many equivalent formulations of the definition, the most relevant for us being that a subgroup of $\Mod(S)$ is convex cocompact if and only if it is finitely generated and the orbit map to the curve graph $\cc(S)$ is a quasi-isometric embedding. 

Despite its success, the theory of convex cocompactness is inhibited by a relative scarcity of examples. There are several constructions of convex cocompact subgroups of the mapping class group, but to date all known examples are virtually free. Two major open questions in the field are whether there exists a convex cocompact surface subgroup of $\Mod(S)$, and whether there exists a purely pseudo-Anosov subgroup that fails to be convex cocompact. Kent and Leininger \cite{KL24} have recently shown the existence of (infinitely many commensurability classes of) purely pseudo-Anosov surface subgroups of $\Mod(S)$, when $S$ is a closed orientable surface of genus $g\geq 4$, but it unknown whether these are convex cocompact.

In the setting of Kleinian groups, convex cocompactness is a restrictive case of the more general phenomenon of \textit{geometric finiteness}. Mosher \cite{MR2264544} suggested in 2006 that there should be an analogous theory of geometrically finite subgroups of mapping class groups, a hope that is finally coming into view now. While there are arguably many potential formulations of what ``geometrically finite'' should mean in this setting, recent work of Dowdall--Durham--Leininger--Sisto \cite{DowdallDurhamLeiningerSisto-ExtensionsII}, Loa \cite{Loa}, and Udall \cite{Udall-combinations_of_PGF} has focused attention on a qualified notion called \emph{reducibly geometrically finite} (RGF), which roughly means $G$ is hyperbolic relative to a collection $\mathcal{H} = \{H_1,\dots,H_n\}$ of reducible subgroups $H_i \le G$ for which the coned off Cayley graph equivariantly and quasi-isometrically embeds into the curve graph $\Cc(S)$; see \Cref{def:RGF}. Recall that a subgroup $H\le \Mod(S)$ is \emph{reducible} if there is a multicurve $\alpha$ on $S$ that is preserved by every element of $H$. 

The goal of this paper is to provide many new examples of RGF subgroups of mapping class groups and to clarify when certain constructions yield RGF subgroups. These examples provide a wealth of different features and can serve as testing ground for the continued development of the theory of geometric finiteness in mapping class groups. Throughout, let $S$ be a connected oriented surface without boundary and with finite complexity $\xi(S)$ (see \Cref{ssub:Surfaces_complexes}) at least 1.

\subsection*{Right-Angled Artin Subgroups}
One important source of interesting subgroups in mapping class groups are the right-angled Artin groups constructed by Koberda \cite{Koberda-RAAG_subgroups_of_MCGs} and Clay--Leininger--Mangahas \cite{CLM-RAAGs_in_MCG}; see also Crisp--Wiest \cite{CrispWeist-RAAGs_in_pure_braid_groups}.
Recall that to each finite simplicial graph $\Gamma$, there is an associated \emph{right-angled Artin group (RAAG)} $A(\Gamma)$ defined by the presentation
\[A(\Gamma)=\left\langle x_1, \dots, x_n \mid [x_i,x_j]=1 \text{ if } (x_i,x_j)\text{ is an edge in }\Gamma \right\rangle \]
with generators $x_1,\dots,x_n$ corresponding to the vertices of $\Gamma$. Given a full subgraph $\Gamma'\subset \Gamma$, we also use $A(\Gamma')\le A(\Gamma)$ to denote the subgroup generated by vertices $x_i\in \Gamma'$.
Right-angled Artin groups interpolate between free groups (when there are no edges) and free abelian groups (when $\Gamma$ is a complete graph). Due to their simple yet flexible formulation, such groups exhibit a rich variety of behaviors and play an essential role throughout geometric group theory,  including in Agol's celebrated resolution of the virtual Haken conjecture \cite{Agol-Virtual_Haken};  see e.g.~\cite{charney,Wise-Riches_to_RAAGs} and the references therein.

Our first theorem addresses the question of determining when these right-angled Artin subgroups are reducibly geometrically finite. For the statement,  consider a list $S_1,\dots, S_n$ of isotopy classes of essential subsurfaces of a surface $S$ and let $\Gamma = \Gamma(S_1,\ldots,S_n)$ be the {\em realization graph} with vertex set $\{S_1,\ldots,S_n\}$ and edges representing disjointness. We say the family is \define{admissible} if for $i\ne j$ the surfaces $S_i\ne S_j$ are non-nested and if $S_j$ is not the annulus about a boundary component of $S_i$. Now take mapping classes $f_1, \dots, f_n$ that are \emph{fully supported} on these subsurfaces, meaning each $f_i$ is either a partial pseudo-Anosov supported on $S_i$ or a Dehn twist power about the core of $S_i$ in the case that $S_i$ is an annulus.

In this setting, Koberda \cite{Koberda-RAAG_subgroups_of_MCGs} showed that for all large $r $ the map $x_i\to f_i^r$ gives an isomorphism between the RAAG $A(\Gamma)$ and the subgroup $\langle f_1^r, \dots, f_n^r\rangle$ of $\Mod(S)$ generated by powers of these elements. (Koberda's result holds more generally 
whenever the collection $f_1,\dots, f_n$ is ``irredundant'', a condition that is implied by our admissibility assumption on the supports $S_1,\dots, S_n$.) This gives a complete \emph{algebraic} description of the subgroup generated by the powers $f_i^r$. 
Under the additional assumption that each $f_i$ is a partial pseudo-Anosov (that is, no $S_i$ is an annulus), Clay, Leininger, and Mangahas \cite{CLM-RAAGs_in_MCG} independently proved that $\langle f_1^r, \dots, f_n^r\rangle$ is isomorphic to $A(\Gamma)$ and moreover that it equivariantly quasi-isometrically embeds into the mapping class group (i.e., it is an \emph{undistorted} subgroup) and the Teichm\"uller space, thereby giving strong \emph{geometric} information about the subgroup. Later Runnels \cite{Runnels_EffectiveGenerationRAAGsinMCG} gave an effective upper bound on the size of the exponent $r$ needed for these results to hold and extended the result about undistortion in $\Mod(S)$ to also allow the $f_i$ to be Dehn twist powers, thereby confirming a speculation made by the authors of \cite{CLM-RAAGs_in_MCG}.

Our first theorem explains precisely when the above construction produces subgroups with the additional property of being reducibly geometrically finite:

\begin{introthm}\label{main-CLM}
Let $f_{1},\ldots, f_{n}$ be mapping classes fully supported, respectively, on an admissible family $\{S_{1},\ldots , S_{n}\}$ of subsurfaces with realization graph $\Gamma$. 
Suppose 
\begin{enumerate}
\item\label{main-thm-cond-subgraph} $\Gamma$ decomposes as the disjoint union  $\Gamma_1
\sqcup \cdots \sqcup \Gamma_m$ of subgraphs, 
with $m \geq 2$; 

\item\label{main-thm-cond-reducible} the subgroup $G_k$ 
of $\Mod(S)$ generated by the elements $f_i$ supported on the vertices of $\Gamma_k$
is reducible for each $k=1,\ldots,m$; and 
\item\label{main-thm-cond-bigdist} 
$d_S(\partial S_\ell, \partial S_j)\geq 3$ for all $S_\ell$ and $S_j$ belonging to 
distinct subgraphs $\Gamma_k$
of $\Gamma$. 
\end{enumerate}
Define a map $\Psi\colon A(\Gamma)\to \Mod(S)$ by $\Psi(x_i) = f_i^{p_i}$ for some exponents $p_i\in \Z$.
Then there exists $N>0$ such that whenever $\abs{p_i}\ge N$ for each $i$, the subgroup 
$\left\langle f_1^{p_1}, \ldots, f_n^{p_n}\right\rangle$
is isomorphic to 
$\Psi(A(\Gamma_{1})) \ast \cdots \ast \Psi(A(\Gamma_{m})) $, and 
is a 	reducibly geometrically finite group with 
respect to the factors $\{\Psi(A(\Gamma_{1})),\ldots, \Psi(A(\Gamma_{m}))\}$.
\end{introthm}

Note that the subgraphs $\Gamma_k$ in \Cref{main-CLM} are not required to be connected and that the reducibility condition in (\ref{main-thm-cond-reducible}) is equivalent to the existence of a simple closed curve that is disjoint from all the subsurfaces $S_\ell$ lying in the subgraph $\Gamma_k$.

By combining with the above mentioned results of Koberda \cite[Theorem 1.1]{Koberda-RAAG_subgroups_of_MCGs}, Clay--Leininger--Mangahas \cite[Theorem 5.2]{CLM-RAAGs_in_MCG}, and Runnels \cite[Theorem 2]{Runnels_EffectiveGenerationRAAGsinMCG}, which applies to the mapping classes $f_1,\dots, f_n$ above, we can strengthen \Cref{main-CLM} to additionally conclude $\langle f_1^{p_1},\dots, f_n^{p_n}\rangle$ is an undistorted right-angled Artin subgroup:

\begin{corollary}
\label{cor:main-CLM}
Under the hypotheses of \Cref{main-CLM}, the number $N$ can be chosen so that 
$\Psi\colon A(\Gamma)\to \Mod(S)$ is an injective q.i.-embedding. In particular, the image $\langle f_1^{p_1},\dots, f_n^{p_n}\rangle$ 
is undistorted in $\Mod(S)$, isomorphic to $A(\Gamma)$, and RGF relative the to the family of RAAG subgroups $\Psi(A(\Gamma_k))\cong A(\Gamma_k)$ for $k = 1,\dots, m$.
\end{corollary}

\begin{remark}
\Cref{main-CLM} is sharp in the sense that all conditions (\ref{main-thm-cond-subgraph})--(\ref{main-thm-cond-bigdist}) are necessary for the conclusion to hold for all large exponents $p_i$ (of course the conclusions may hold for some smaller exponents as well). This is because the definition of RGF requires hyperbolicity relative to a family of reducible subgroups.
Thus (\ref{main-thm-cond-subgraph}) is necessary in order for the RAAG $A(\Gamma)\cong \Psi(A(\Gamma))$ to be relatively hyperbolic, which by \cite[Proposition 1.3]{BehrstockDrutuMosher-RelHyp} is known to hold if and only if the defining graph $\Gamma$ is disconnected,  and (\ref{main-thm-cond-reducible}) is necessary for the subgroups $\Psi(A(\Gamma_k))$ to be reducible. Finally, in order for the coned-off Cayley graph to quasi-isometrically embed into the curve graph, all elements that are not conjugate into one of the reducible subgroups $\Psi(A(\Gamma_j))$ must act loxodromically on the curve graph. In particular, the product $f_\ell^{p_\ell} f_j^{p_j}$ must be pseudo-Anosov when the supports $S_\ell$ and $S_j$ belong to distinct subgraphs $\Gamma_k$, which is not guaranteed without the separation condition $d_S(\partial S_\ell, \partial S_j)\ge 3$ of (\ref{main-thm-cond-bigdist}).
\end{remark}

\subsection*{Combinations of reducible subgroups}

\Cref{main-CLM} gives conditions for when a collection of reducible elements will generate, after passing to sufficiently high powers, an RGF subgroup of $\Mod(S)$. Our next theorems expand on this in two orthogonal directions: firstly by generalizing from reducible elements $f_i$ to reducible subgroups $H_i$, and secondly by removing the necessity of raising to powers. Note that raising a reducible element $f_i$ to a power $f_i^{p_i}$ corresponds to passing to a finite index subgroup $\langle f_i^{p_i}\rangle$ of the reducible group $\langle f_i\rangle$ generated by the element. 
Thus in this context, raising the elements $f_i$ to powers is roughly analogous to passing to finite-index subgroups of the reducible groups $H_i$. This motivates:

\begin{question}
\label{quest:more_general}
Given a list $H_1,\dots, H_m$ of reducible subgroups of $\Mod(S)$:
\begin{enumerate}
\item\label{quest:no_powers} Under what conditions do the $H_i$ generate an RGF subgroup $\langle H_1,\dots, H_m\rangle$?
\item\label{quest:finite-index} Under what conditions can one pass to finite index subgroups $H'_i\le H_i$ that generate an RGF subgroup?
\end{enumerate}
In particular, what additional hypotheses are needed to guarantee the conclusion of \Cref{main-CLM} without raising the elements $f_i$ to powers?
\end{question}

In answering~\Cref{quest:more_general} we will formulate our conditions in terms of how the reducible subgroups are situated in the curve graph $\cc(S)$ of the surface.
To state these, we first associate to each reducible subgroup $H\le \Mod(S)$ a \textit{canonical reducing system} $\partial H$ (\Cref{def:canonical_reducing_sys}), which is the multicurve consisting of all simple closed curves with finite $H$--orbit and which are disjoint from all of other curves with finite $H$--orbit. 
This is a direct generalization to subgroups of the well-studied canonical reducing systems of reducible elements, and we use ideas from \cite{HandelThurston} to show $\partial H$ is non-empty whenever $H$ is infinite and reducible; see \Cref{lem:reducing_multicurve}. 
We then say a family $\{H_1,\dots, H_m\}$ of reducible subgroups is \define{$D$--separated} (\Cref{def:separated}) if their canonical reducing systems have pairwise distance at least $D$ in the curve graph; that is $d_S(\partial H_i, \partial H_j) \ge D$ for all $i\ne j$. We also say the family is \define{$A$--misaligned} (\Cref{def:misaligned}) if for all distinct indices $i,j,k$ the Gromov product ($\partial H_i \mid \partial H_k)_{\partial H_j}$ is at least $A$; this roughly means that $\partial H_j$ lies at least distance $A$ from the geodesic joining $\partial H_i$ and $\partial H_k$.

In \Cref{s:separability} we prove the following result, which addresses \Cref{quest:more_general} (\ref{quest:no_powers}) above:

\begin{introthm}\label{main-Loa}
There exist constants $D,A>0$ such that if $\Hc=\{H_1, \dots, H_n\}$ is a $D$--separated and $A$--misaligned family of torsion-free reducible subgroups of $\Mod(S)$, then $\langle H_1, \dots, H_n \rangle$ is isomorphic to $H_1\ast \dots \ast H_n$ and RGF relative to $\mathcal{H}$.
\end{introthm}
This generalizes a recent theorem of Loa \cite[Theorem 1.1]{Loa}, 
which proves that if $H_\alpha$ and $H_\beta$ are abelian subgroups consisting of multitwists supported on multicurves $\alpha$ and $\beta$, then $\langle H_\alpha, H_\beta\rangle$ is a free product $H_\alpha\ast H_\beta$ and \emph{parabolically geometrically finite (PGF)} provided $\alpha$ and $\beta$ are sufficiently far apart in the curve graph. 
Recall that PGF is a more restrictive version of RGF requiring the reducible subgroups to be virtual multitwist groups; see \Cref{def:RGF}.
Note also that our misalignment assumption is vacuous when there are only two subgroups in the family $\mathcal{H}$.
Thus \Cref{main-Loa} generalizes \cite[Theorem 1.1]{Loa} in two ways: by allowing for arbitrary torsion-free reducible subgroups, rather than virtual multitwist groups, and by accommodating families of $3$ or more subgroups. 
 We will see in \Cref{examples} that the $A$--misaligned and torsion-free assumptions are both necessary in \Cref{main-Loa}.

In the setup of \Cref{main-CLM}, it is not hard to see that $d_S(\partial S_j, \partial G_k)\le 1$ for each subsurface $S_j$ belonging to the subgraph $\Gamma_k$. Thus, \Cref{main-Loa} allows us to strengthen the conclusion of \Cref{main-CLM} in certain circumstances:

\begin{corollary}
There exist $A, D >0$ so that, under the hypotheses of \Cref{main-CLM}, if the subsurfaces satisfy $d_S(\partial S_\ell, \partial S_j)\ge D$ and $(\partial S_i \mid \partial S_\ell)_{\partial S_j} \ge A$ whenever $S_i, S_j, S_\ell$ belong to distinct subgraphs $\Gamma_k$ of $\Gamma$, then  $\langle f_1,\dots, f_n\rangle$ itself is RGF with respect to $\{G_1,\dots, G_m\}$.
\end{corollary}

When the reducing systems $\partial H_i$ are nearby in the curve graph, one cannot expect the subgroup $\langle H_1,\dots, H_m\rangle$ to be RGF (indeed, the constant $D$ from \Cref{main-Loa} is ineffective and presumably quite large).
However, our final theorem, which answers \Cref{quest:more_general}(\ref{quest:finite-index}),  shows that as long the family is merely $5$--separated,
 one can always achieve reducible geometric finiteness by passing to finite index subgroups.

\begin{introthm}\label{main-subgroups}
Let $G_1, \dots, G_m,$ be infinite reducible subgroups of $\Mod(S)$ that are pairwise $5$--separated. Then there are finite index subgroups $G_i'\le G_i$ so that for any further subgroups $H_i \le G_i'$ which are still infinite, the group $\langle H_1, \dots, H_n \rangle$ is isomorphic to $H_1 \ast \dots \ast H_m$ and is  RGF relative to $\{H_1,\dots, H_m\}$.
\end{introthm} 

Since raising to powers is analogous to passing to finite index, \Cref{main-subgroups} is closely related to \Cref{main-CLM} and, in fact, easily implies a slight variation on this result; this is accomplished in \Cref{thm:weak-CLM}.

\begin{remark}
\label{rem:without_boundary}
We work in the context of surfaces without boundary. When a surface $S$ has boundary, the mapping class group $\Mod(S)$ has a nontrivial center 
consisting of twists about boundary components.  So, the center acts trivially on the curve graph and its action is not captured by orbits in $\cc(S)$. Consequently, when $S$ has nonempty boundary, 
the coned-off Cayley graph of $G\le \Mod(S)$ cannot quasi-isometrically embed into the curve graph $\cc(S)$ 
unless $G$ has trivial intersection with the center.
In order to extend our results to this setting of surfaces with boundary, it is necessary to assume the reducible subgroups $H_i$ intersect the center trivially.

\end{remark}

\subsection*{Organization of the paper} In \Cref{background} we discuss the notation we will use throughout the paper and review the necessary background on hyperbolic spaces, curve graphs, subsurface projections, mapping class groups and the distance formula.
In \Cref{sec:reducible_subgroups} we discuss reducible subgroups and their canonical reducing systems, while in \Cref{ssub:RGF} we recall the notion of relative hyperbolicity and the definition of reducibly geometrically finite subgroups of the mapping class group. In \Cref{s:cayley_trees} we describe the Bass--Serre tree associated to a free product and define an equivariant map into the curve graph that will be used in the proofs of our main theorems. In \Cref{s:displacing} we prove \Cref{main-subgroups} as consequence of  \Cref{thm:displacing}, where we use in a crucial way the notion of ``displacing'' families, and in \Cref{sec:RAAGS} we prove \Cref{main-CLM} by using a technical result (\Cref{cor:general_projections_persist}) related to the Behrstock inequality. In \Cref{s:separability} we prove \Cref{main-Loa} generalizing ideas from Loa \cite{Loa} and where we use crucially the notions of separability and misalignment. Finally in \Cref{examples} we discuss examples illustrating our constructions and explaining why the assumptions of our theorems are necessary.

\subsection*{Acknowledgments} We thank Chris Leininger and Brian Udall for insightful conversations. We also thank the referee for helpful comments. S.D.~was partially supported by U.S.~National Science Foundation (NSF) grants DMS-2005368 and DMS-2405061. H.H.~was supported by NSF grant DMS-2303365.	S.M.~was partially supported by U.S. National Science Foundation (NSF) grant DMS-1848346 (NSF CAREER). The authors started their collaboration at the ``Collaborative Research Project Workshop'' organized with support from the U.S. NSF grant DMS-1839968 (NSF RTG).

\section{Background and notation}
\label{background}

\subsection{Hyperbolic metric spaces}
\label{S:hyperbolic}
In this subsection we recall the fundamentals of coarsely hyperbolic metric spaces, in the sense of Gromov. 
Throughout, $(X,d)$ denotes a geodesic metric space and all triangles are taken to have geodesic edges.

For $x,y,z \in X$, the \textit{Gromov product} of $x$ and $y$ with respect to $z$ is 
\[ (x \mid y)_z= \frac{1}{2}(d(x,z)+ d(y,z)- d(x,y)). \]
We note the trivial fact (coming from the triangle inequality) that
\begin{equation}
\label{eqn:GP-stability}
\abs{(y\mid  x)_z - (x\mid w)_z} \le d(y,w). 
\end{equation}

We define $(X,d)$ to be a {\em hyperbolic metric space} \cite[Chapitre 1, D\'efinition 1.4]{CDP} if there exists a $\delta>0$ such that the following holds for all $x,y,z,w\in X$:
\begin{equation}
\label{eqn:hyperblicity}
(x\mid y)_w\geq \min\{(x\mid z)_w,(y\mid z)_w\}-\delta.
\end{equation}
In this case, we say $(X,d)$ is $\delta$-hyperbolic. 
 
One consequence of hyperbolicity is that if $(X,d)$ is $\delta$-hyperbolic then geodesic triangles are {\em $4\delta$-thin}, meaning for any geodesic triangle $T$, any edge of $T$ is contained in the $4\delta$-neighborhood of the other two edges of $T$ \cite[Proposition 3.6]{CDP}. 
Another consequence is that \textit{inner triangles} of (geodesic) triangles are small, in the following sense. Let $T$ be a triangle with vertices $x,y,z$ and edges $[x,y],[y,z],$ and $[x,z]$. Then, as pictured in \Cref{fig:tripod}, there exist unique points $a\in [x,y]$, $b\in [y,z]$, and $c\in [x,z]$ such that $d(a,x)=d(c,x)$, $d(a,y)=d(b,y)$, and $d(c,z)=d(b,z)$ (see discussion in \cite{BH} page 408 before Definition 1.16).
An {\em inner triangle} of $T$ is a geodesic triangle with vertices $a,b,c$. If $(X,d)$ is a $\delta$-hyperbolic metric space, then each edge of an inner triangle has length $\leq 4\delta$. For a proof, see \cite[Chapter III.H Proposition 1.17]{BH} or \cite[Chapitre 1, Proposition 3.2]{CDP}.

 \begin{figure}[h]
 	\centering
 	\begin{tikzpicture}[scale=.7,rotate=65]
  \begin{scope}[thick]
  \draw (0,0) coordinate (x);
  \draw (1,7) coordinate (z);
  \draw (5,0) coordinate (y);
  \draw (x) to [out=50,in=-80] coordinate[pos=.3] (c) (z);
  \draw (x) to[out=50, in=150] coordinate[pos=.4] (a) coordinate[pos=.45] (q) (y);
  \draw (z) to[out=-80, in=150] coordinate[pos=.6] (b) coordinate[pos=.8]
  (w) (y);
  \end{scope}

  \draw[thick,blue,dashed] (z) -- coordinate[pos=.65] (gp) (q);
  \path (z) -- coordinate[pos=.4] (label) (y);

  \draw[->,blue,shorten >=3pt] (label) -- (gp);
  \path (label) node[above,blue] {$(x\mid y)_z$};
  \begin{scope}[]
  \draw (c) to[out=-10,in=90] (a);
  \draw (b) to[out=210,in=90] (a);
  \draw (c) to[out=-10,in=210] (b);
  \end{scope}

  \def\s{.05}
  \draw[fill] (x) circle (\s) node[below right] {$x$};
  \draw[fill] (y) circle (\s) node[ right] {$y$};
   \draw[fill] (z) circle (\s) node[left] {$z$};
  \draw[fill] (b) circle (\s) node[above] {$b$};
  \draw[fill] (c) circle (\s) node[below left] {$c$};
  \draw[fill] (a) circle (\s) node[below right] {$a$};
  
\end{tikzpicture}
 	\caption{In a hyperbolic metric space, the geodesic triangle pictured here has inner triangle with vertices $a,b,c$ within distance $4\delta$ of each other. The Gromov product $(x\mid y)_z$ approximates the distance from $z$ to any geodesic from $x$ to $y$ within $4\delta$.} 
 	\label{fig:tripod}
 \end{figure} 
 
It is easy to see that $(x\mid  y)_z$ is always bounded above by the minimal distance $d(z,[x,y])$ from $z$ to any geodesic joining $x$ and $y$. Indeed, for any $t\in [x,y]$ the triangle inequality immediately gives $(x\mid  y)_z \le d(z,t)$.
A fundamental feature of hyperbolicity is a coarse inequality in the other direction as well, so that the Gromov product coarsely measures how close a side of a geodesic triangle is to the opposite vertex. So in hyperbolic spaces we can see this as a geometric interpretation of the quantity. Precisely, if $X$ is $\delta$-hyperbolic, then for all $x,y,z\in X$ one has:
\begin{equation}
\label{eqn:gromov-product-to-geod}
  d(z,[x,y])-4\delta \leq (x\mid y)_z\leq d(z,[x,y]).
\end{equation}
For a proof, see \cite[Chapitre 3, Lemme 2.7]{CDP}. Notice in particular that when the Gromov product $(x\mid y)_z$ is small, it means that the point $z$ is close to the geodesic $[x,y]$. We will use this point of view to give us a geometric interpretation of the local-to-global principle we will state below. 

It is well known that hyperbolic metric spaces have the ``local-to-global'' property, meaning that local geodesics are in fact global quasi-geodesics. This phenomenon is usually formulated in terms of parameterized quasi-geodesics, see, for instance, \cite[Chapitre 3]{CDP}. For our purposes, it will suffice to use the following formulation in terms of Gromov products and the reverse triangle inequality. While this basic idea is well-known, we include a short proof for completeness. 
A geometric interpretation of this statement is that if a piece-wise geodesic path is built from long geodesic segments with angle between segments close to $\pi$, then the path is a quasi-geodesic and the concatenation points are close to the geodesic between the extremes.

\begin{lemma}[Local-to-Global]
\label{lem:local-to-global}
Let $X$ be $\delta$-hyperbolic and suppose $x_0,\dots,x_n\in X$ are such that $(x_{i-1}\mid  x_{i+1})_{x_i}\le A$ and $d(x_i,x_{i+1})>4A+24\delta$ for each $i$. Then any geodesic joining $x_0$ and $x_n$ passes within $D=A+6\delta$ of each point $x_i$ and 
\[d(x_0,x_n) \ge d(x_0,x_1) + \dots + d(x_{n-1},x_n) - 2D(n-1).\]
In particular, $d(x_0,x_n) \le \sum_{j=1}^n d(x_{j-1},x_j) \le 2 d(x_0,x_n)$.
\end{lemma}
\begin{proof}
The proof is by induction on $n$. The base case $n=1$ is trivial, and the case $n=2$ follows from (\ref{eqn:gromov-product-to-geod}), which provides a point $y_1\in [x_0,x_2]$ so that $d(y_1,x_1)\leq A+4\delta\leq D$ and hence
\[d(x_0,x_2) = d(x_0,y_1) + d(y_1,x_2) \ge d(x_0,x_1) + d(x_1,x_2) - 2D.\]

Now fix $n> 2$ and any index $0<i<n$. By induction, the geodesic $[x_i,x_n]$ contains a point $y$ within $D$ of $x_{i+1}$. Since $d(x_i,y) + d(y,x_n) = d(x_i,x_n)$, using (\ref{eqn:GP-stability}) we find that
\[(x_{i+1}\mid x_n)_{x_i} \ge (y\mid x_n)_{x_i} - D = d(x_i,y) - D \ge d(x_i,x_{i+1}) - 2D > A+2\delta.\]
Similarly $(x_0\mid x_{i-1})_{x_i} >A+2\delta$. Two applications of (\ref{eqn:hyperblicity}) now gives
\[A+2\delta \ge (x_{i-1}\mid x_{i+1})_{x_i} + 2\delta
\ge \min\{(x_{i-1}\mid x_0)_{x_i}, (x_{0}\mid x_{n})_{x_i}, (x_n\mid x_{i+1})_{x_i}\}.\]
Since the first and last quantities in the minimum have been seen to be large, this is only possible if $(x_0\mid x_n)_{x_i} \le A+2\delta$ which, by (\ref{eqn:gromov-product-to-geod}), provides a point $y_i\in [x_0,x_n]$ with $d(y_i,x_i)\le A+6\delta=D$. Hence, again by induction, we find that
\[d(x_0,x_n) \ge d(x_0,x_i) + d(x_i,x_n) - 2D
\ge -2D(n-1) + \sum_{j=1}^n d(x_{j-1},x_j).\]
The final claim is now an immediate consequence: The upper bound on $d(x_0,x_n)$ is just the triangle inequality, and the lower bound holds since the hypotheses ensure $d(x_{j-1},x_j) \ge 4D$ and thus $d(x_{j-1},x_j) - 2D \ge \frac{1}{2}d(x_{j-1},x_j)$ for each $j$.
\end{proof}

\subsection{Relative hyperbolicity}
\label{sec:rel_hyp}

Let $G$ be a finitely generated group and $\Gamma$ its Cayley graph with respect to some finite generating set $X$.  Given a finite list of subgroups $\{H_1, \ldots, H_l\}$, the associated \textit{coned-off Cayley graph} $\widehat{\Gamma}(G, \{H_1, \ldots, H_l\})$ is the graph obtained from $\Gamma$ by adding a new ``coset vertex'' $gH_i$ for every left coset of each subgroup $H_i$ and by adding edges of length $\frac{1}{2}$ from $gH_i$ to each element in $gH_i$. Up to quasi-isometry, this graph does not depend on the generating set $X$.

In Farb's original paper on relatively hyperbolic groups \cite{Farb}, $G$ is declared to be \define{hyperbolic relative to the subgroups $H_1,\dots,H_l$} if the coned-off Cayley graph $\widehat{\Gamma}(G;\{H_1,\dots,H_l\})$ is Gromov hyperbolic and satisfies a technical \define{bounded coset penetration} (BCP) property dictating how quasi-geodesics intersect the cosets $gH_i$. Bowditch later gave an equivalent definition that is better suited to our purposes:

\begin{definition}[{Bowditch \cite{Bowditch-RelativelyHyperbolic}}]
\label{def:rel_hyp}
A group $G$ is \define{hyperbolic relative} to a collection $\{H_1, \ldots, H_l\}$ of finitely generated subgroups if it acts on a connected hyperbolic graph $T$ with only finitely many orbits of edges such that:
\begin{itemize}
\item  each edge has trivial stabilizer and is contained in only finitely many circuits of length $n$ for each $n\in \mathbb{N}$, and
\item infinite vertex stabilizers are exactly the conjugates of the subgroups $H_i$.
\end{itemize}
\end{definition}

\begin{remark}
\label{rem:bowditch_tree_qi_coned-off_Cayley}
In \Cref{def:rel_hyp} it is not hard to see that $T$ is $G$--equivariantly quasi-isometric to the Coned-off Cayley graph $\widehat{\Gamma}$. Indeed, the action gives an equivariant, Lipschitz orbit map $\Gamma\to T$ which can be extended to $\widehat{\Gamma}$ be sending each coset vertex $gH_i$ to the unique vertex of $T$ with stabilizer $gH_i g^{-1}$. A quasi-isometric inverse is obtained by sending a vertex of $T$ with stabilizer $gH_i g^{-1}$ back to $gH_i$.
\end{remark}

	\subsection{The curve graph} \label{ssub:Surfaces_complexes}
	For the purposes of this paper a \emph{surface} is an orientable $2$-dimensional manifold with compact boundary and finitely generated fundamental group. Connected surfaces are classified by the triple $(g,n,b)$, where $g$ is the genus of the surface, $n$ is the number of punctures, and $b$ is the number of boundary components. The \emph{complexity} of a surface is $\xi(S)=3g+n+b-3$. A closed curve in a surface $S$ is a continuous map $\mathbb{S}^1 \arr S$. The curve is \emph{simple} if the map is an embedding, and \emph{essential} if it is not homotopic to a point, a puncture, or a  boundary component. We consider curves as equivalent if they differ by free homotopy or precomposition with a (possibly orientation reversing) homeomorphism of $\mathbb{S}^1$. For the remainder of the paper we use the term \emph{curves} to mean the equivalence class of an essential simple closed curve. Curves are said to be \emph{disjoint} if they have disjoint representatives and are said to \define{intersect} otherwise. A \textit{multicurve} is a disjoint collection of distinct curves. 

The \emph{curve graph} $\Cc(S)$ of a connected surface $S$ with $\xi(S) \ge 1$ is the graph with vertices labeled by 
curves in $S$. When $\xi(S)\ge 2$, there is an edge between two vertices if the curves are disjoint.
In the case $\xi(S) = 1$ the edge condition is modified so that vertices span an edge if the curves have representatives intersecting once in the case $g = 1$, or twice in the case $g = 0$.
This modification ensures $\Cc(S)$ is connected for a surface of any genus.

For an annulus, i.e.~ the connected surface $S$ with $g=n=0$ and $b=2$, the curve graph $\Cc(S)$ is defined to have vertices given by homotopy classes, rel endpoints, of embedded arcs $[0,1]\to S$ with endpoints lying in distinct boundary components of $S$. Two vertices are then joined by an edge if the arcs have representatives with disjoint interiors. It is not hard to see that in this case $\Cc(S)$ is connected and quasi-isometric to $\Z$.

For any surface, the curve graph is a metric space with each edge having length 1. 
An essential fact of the curve graph in our setting is hyperbolicity: 

\begin{theorem}[{Masur--Minsky \cite[Theorem 1.1]{MasurMinsky-CurveComplexI}}]
\label{thm:curve-graph-hyp}
If $S$ is an annulus or a connected surface with $\xi(S)\ge 1$, then the curve graph $\Cc(S)$ is Gromov hyperbolic.
\end{theorem}

\subsection{Subsurface projections} \label{ssub:Projections}
By a \emph{subsurface} of $S$, we mean a subset $Y\subset S$ that is also a surface. 
We always assume the subsurface $Y$ is connected and \emph{essential}, meaning each of its boundary components is either a boundary component of $S$ or an essential curve in $S$.
We write $\partial Y$ for the multicurve consisting of boundary components of $Y$ that are essential in $S$. 
We shall also assume $Y$ is an annulus or has $\xi(Y)\ge 1$, so that $Y$ has its own curve graph $\Cc(Y)$. 

Given a curve $\gamma$ on $S$, its \emph{projection} to $Y$ is a set $\pi_Y(\gamma)$ of curves in $Y$ defined following \cite{MasurMinsky-CurveComplexII}: Firstly, $\gamma$ and $Y$ are said to be \emph{disjoint} if they have disjoint representatives; in this case $\pi_Y(\gamma) $ is the empty set. Otherwise $\gamma$ and $Y$ are said to \emph{intersect} and we may realize $\gamma$ and $\partial Y$ in minimal position so that $\gamma\cap Y$ is a nonempty collection of curves and proper arcs in $Y$. We now define projection to $Y$ when $Y$ is not an annulus. For each component $\gamma'$ of $\gamma\cap Y$, the boundary $\partial N_\epsilon(\gamma' \cup \partial Y)$ of a sufficiently small regular neighborhood of $\gamma'\cup \partial Y$ yields a disjoint collection of simple (but possibly nonessential) closed curves. 
We then define $\pi_Y(\gamma)$ to be the set of all such curves that are essential in $Y$ and therefore vertices of $\mathcal{C}(Y)$. Since $Y$ is non-annular, $\pi_Y(\gamma)\neq\emptyset$   when $\gamma$ is not disjoint from $Y$.

When $Y$ is an annulus, the definition of subsurface projection requires a bit more care. The issue is that in this case, the two boundary components of $Y$ are isotopic and different choices for representatives of these boundary components (and therefore for $Y$) can alter the isotopy class of the projection of an arc to $\mathcal{C}(Y)$. To rectify this, we equip $S$ with a complete hyperbolic structure and consider the cover $p: S_Y \rightarrow S$
associated to $\pi_1(Y)$. This inherits a hyperbolic structure from $S$ and admits a compactification $\overline{S_{Y}}$ coming from adding the (quotient of the) boundary at infinity of $\tilde{S} = \mathbb{H}^{2}$. The projection $\pi_{Y}(\gamma)$ is then defined to be the compactification of $p^{-1}(\gamma)$ in $\overline{S_{Y}}$: this is a collection of pairwise disjoint arcs in the compact annulus $\overline{S_{Y}}$ and so determines a (possibly empty) diameter $1$ subset of the annular complex $\mathcal{C}(Y)$.

The definition of subsurface projection is extended to arbitrary sets of curves on $S$ by defining $\pi_Y(A)$ to be the union of the projections of all curves in $A$.  The \textit{subsurface distance} between any two subsets $A,B$ of $\Cc(S)$ is then defined to be the diameter of their projections to $Y$: 
\[ d_{Y}(A,B) \colonequals \mbox{diam}_{\mathcal{C}(Y)}\left(\pi_{Y}(A) \cup \pi_{Y}(B)\right). \]
To ease notation, we sometimes use the shorthand $d_Y(W, V) \colonequals d_Y(\partial W,\partial V)$ when $W, V$ are subsurfaces of $S$. We also note the following basic fact from \cite[Lemma 2.2]{MasurMinsky-CurveComplexII}:
\begin{equation}
\label{eqn:multicure_bdd_projection}
d_Y(\alpha,\beta) \le 2\quad\text{ for any subsurface $Y$ and disjoint multicurves $\alpha,\beta$ on $S$}.
\end{equation}

The following \emph{Bounded Geodesic Image Theorem (BGIT)} of Masur and Minsky will play a fundamental role in our arguments:

\begin{theorem}[Masur--Minsky {\cite[Theorem 3.1]{MasurMinsky-CurveComplexII}}] \label{thm:BddGeodesicImage}
There is a constant $\bgit$, depending only on $\xi(S)$, such that for any proper essential subsurface $Y$ of $S$ and any geodesic $g$ in $\Cc(S)$, if $\pi_Y(v)\ne \emptyset$ for each vertex $v$ in $g$, then $\diam_Y(g) \le \bgit$.
\end{theorem}

We will typically use this theorem by applying its converse to compute lower bounds for distances in the curve graph of $S$. The converse essentially states that if the projection of two curves are sufficiently far apart in $\Cc(Y)$, then the geodesic between them in $\Cc(S)$ makes a pit stop near $\partial Y$.

We shall also make use of the following inequality due to Behrstock. We say that subsurfaces $X$ and $Y$  \define{overlap}, denoted $X\pitchfork Y$, if each projection $\pi_X(\partial Y)$ and $\pi_Y(\partial X)$ is nonempty. Note that this is implied by $d_S(X,  Y) \ge 2$.

\begin{theorem}[{Behrstock \cite{Behrstock-AsymptoticGeometry}; see also \cite[Lemma 2.13]{Mangahas-ShortWordpAs}}]
\label{thm:behrstock}
There is a constant $\bsi =10$ so that for any non-overlapping subsurfaces $X,Y,Z$ of $S$ one has
\[ d_{Y}( X,  Z)\ge \bsi \implies \max\{d_X( Y,  Z), d_Z( X, Y)\} < \bsi.\]
\end{theorem}

This has the following important consequence. The basic idea 
of \Cref{cor:projections_persist}(\ref{item:persistence}) below 
is well-known to experts and has appeared in different contexts in the literature; see for example \cite[Lemma 5.2]{Mangahas-ShortWordpAs} or \cite[Lemma 4.6]{BBFS-AcylindricalProjectionComplexes}. As we need a precise formulation that also speaks to accumulated distance in the curve graph (item (\ref{item:dist_at_least_3}) below), we include a full proof. Note that the statement is also true for bi-infinite sequences $\{Y_n\}_{n \in \mathbb{Z}}$.

\begin{corollary}\label{cor:projections_persist}
Let $Y_1,\dots,Y_n$ be subsurfaces of $S$ such that $Y_i\pitchfork Y_{i+1}$ and  $d_{Y_j}(Y_{j-1},Y_{j+1})\ge \bgit+3\bsi$ for all $i,j$. Then:
\begin{enumerate}
\item \label{item:persistence} $Y_1,\dots,Y_n$ pairwise overlap and $d_{Y_j}(Y_i ,Y_k) \ge \bgit+\bsi$ for all $i < j < k$.
\item \label{item:dist_at_least_3}  $d_S(Y_{i}, Y_{\ell}) \ge d_S(Y_{j} , Y_{k})$ whenever $i \le j < k \le \ell$.
\end{enumerate}
In particular, if $d_S(Y_j, Y_{j+1})\ge 3$ for each $j$, then $d_S(Y_1, Y_n)\ge n-1$.
\end{corollary}
\begin{proof}
The proof of claim (\ref{item:persistence}) is morally the same as for \Cref{lem:local-to-global}. 
We induct on $n$ assuming the statement holds for a sequence of $n-1$ subsurfaces, with cases $n\le 2$ being trivial.
Given  $Y_1, \ldots Y_n$ and indices $i < j < k$, the subsequences $Y_i,\dots, Y_j$ and $Y_j, \dots, Y_k$ satisfy the  induction hypothesis and hence, the conclusion. 

We claim that $d_{Y_j}(Y_i, Y_{j-1}) \le \bsi$: Indeed, if $i=j-1$ this follows trivially from
\Cref{eqn:multicure_bdd_projection}, and if $i < j-1$ then induction ensures $d_{Y_{j-1}}(Y_i, Y_j)\ge \bsi$ so that the claim follows from \Cref{thm:behrstock}. Similarly we could also show that $d_{Y_j}(Y_{j+1},Y_k)\le \bsi$. 

From that we can conclude the desired lower bound
\begin{align*}
	d_{Y_j}(Y_i, Y_k) &\ge d_{Y_j}(Y_{j-1}, Y_{j+1}) - d_{Y_j}(Y_{j-1}, Y_{i})- d_{Y_j}(Y_{k}, Y_{j+1})\\
	&\ge d_{Y_j}(Y_{j-1}, Y_{j+1})- 2\bsi \\
	&\ge \bgit+\bsi.
\end{align*}
Note this implies $d_S(Y_i, Y_k)\ge 2$ and thus $Y_i\pitchfork Y_k$, since otherwise (\ref{eqn:multicure_bdd_projection}) would imply $d_{Y_j}(Y_i,Y_k)\le 2<\bgit+\bsi$. Hence $Y_1,\dots, Y_n$ pairwise overlap.

For the proof of claim (\ref{item:dist_at_least_3}) it suffices to consider the case $d_S(Y_{j}, Y_k)\ge 3$, since, if $d_S(Y_j, Y_k)=2$, the above already shows $d_{S}(Y_{i}, Y_{\ell})\ge 2$. Consider indices $i \le j < k \le \ell$ and choose components $\alpha\in \partial Y_{j}$ and $\beta\in \partial Y_{k}$ with $d_S(\alpha,\beta) = d_S(Y_{j},Y_{k})$.  If $i = j$ we set $\alpha' = \alpha$, and if $i < j$ then we may choose a component $\alpha'\in \partial Y_{i}$ that projects to $Y_{j}$. Similarly choose $\beta' = \beta$ if $k= \ell$ and otherwise choose $\beta'\in \partial Y_{\ell}$ projecting to $Y_{k}$. We first claim that $\beta'$ intersects $\alpha$ and hence projects to $Y_j$. If $\beta'=\beta$ this is clear, and otherwise $\alpha,\beta'$ both project to $Y_{k}$ and hence by claim (\ref{item:persistence}) and (\ref{eqn:multicure_bdd_projection}) we have
\begin{align*}
	d_{Y_{k}}(\alpha, \beta') &\ge d_{Y_{k}}(Y_{j}, Y_{\ell}) - d_{Y_{k}}(Y_{j}, \alpha)- d_{Y_{k}}(Y_{\ell}, \beta')\\
	&\ge d_{Y_{k}}(Y_{j}, Y_{\ell}) -4\\
	&\ge \bgit + \bsi - 4.
\end{align*}
which, again by (\ref{eqn:multicure_bdd_projection}), is incompatible with $\alpha',\beta$ being disjoint. 

We next claim the geodesic $[\alpha',\beta']$ in $\Cc(S)$ passes through a curve $\alpha_0\notin\{\alpha',\beta'\}$ disjoint from $\alpha$. Indeed, if $\alpha' = \alpha$ we simply take $\alpha_0$ to be the next vertex along the geodesic (this is valid since we have seen $d_S(\alpha,\beta')\ge 2$). Otherwise, by claim (\ref{item:persistence}) and (\ref{eqn:multicure_bdd_projection}), we have
\[d_{Y_j}(\alpha', \beta') \ge d_{Y_j}(Y_{i},Y_{\ell})-4 \ge \bgit+\bsi-4 > \bgit\]
so that \Cref{thm:BddGeodesicImage} provides a curve $\alpha_0\notin\{\alpha',\beta'\}$ that fails to project to $Y_{j}$ and so is disjoint from $\alpha$. Symmetrically, $[\alpha',\beta']$ contains a curve $\beta_0\notin \{\alpha',\beta'\}$ disjoint from $\beta$. Therefore we conclude:
\[d_S(Y_{i},Y_{\ell}) \ge d_S(\alpha',\beta') \ge 2 + d_S(\alpha_0,\beta_0)\ge d_S(\alpha,\beta) = d_S(Y_j, Y_k),\]
where the second inequality follows since $\alpha_0$, $\beta_0$ are interior vertices of $[\alpha' , \beta']$.

Finally, suppose $d_S(Y_{j},Y_{j+1})\ge 3$ for each $j$ and consider a geodesic $\gamma$ realizing $d_S(Y_1,Y_n)$. Since for each $1 < j < n$ we have $d_{Y_j}(Y_1, Y_n)\ge \bgit+\bsi$, \Cref{thm:BddGeodesicImage} ensures the geodesic passes through a curve $\gamma_j$ disjoint from $Y_j$.  These curves are necessarily distinct, since $\gamma_i=\gamma_k$ would imply $d_S(Y_i, Y_k)\leq 2$ in violation of (\ref{item:dist_at_least_3}).  Hence the geodesic passes through the $n-2$ distinct vertices $\gamma_2,\dots,\gamma_{n-2}$ and so has length at least $n-1$.
\end{proof}

We shall also need a variation on this that allows for disjoint subsurfaces within the collection.
Given a list of subsurfaces $Y_1,\dots, Y_n$ of $S$, let:
\begin{itemize}
	\item $\iota(j)$ denote the largest index less than $j$ so that $Y_{\iota(j)} \pitchfork Y_j$, and
	\item $\tau(j)$ the smallest index greater than $j$ so that $Y_j\pitchfork Y_{\tau(j)}$.
\end{itemize}
Note that for a given $j$, its predecessor $\iota(j)$ or successor $\tau(j)$ may not exist.

\begin{corollary}\label{cor:general_projections_persist}
Assume $Y_1,\dots, Y_n$ satisfies the property that $d_{Y_j}(Y_{\iota(j)},Y_{\tau(j)})\ge \bgit+6\bsi$ for all $j \in \{1, \dots, m\}$ such that $\iota(j)$ and $\tau(j)$ both exist. Then for any subsequence $Y_{\sigma(1)},\dots, Y_{\sigma(m)}$ satisfying $Y_{\sigma(j)}\pitchfork Y_{\sigma(j+1)}$ for each $j$ we have:
\[
d_{Y_{\sigma(j)}}(Y_{\sigma(i)}, Y_{\sigma(k)}) \ge \bgit+3\bsi\qquad\text{whenever $\sigma(i) < \sigma(j) < \sigma(k)$}.
\]
\end{corollary}
\begin{remark}
\label{rem:general-to-special-persitence}
It follows that the sequence $W_1,\dots, W_m$, where $W_i = Y_{\sigma(i)}$, satisfies the hypotheses of \Cref{cor:projections_persist} and thus also all of its conclusions.
\end{remark}

\begin{proof}
	We first prove the special case of a subsequence $Y_{\sigma(1)}, Y_{\sigma(2)}, Y_{\sigma(3)}$ of length $m =3$. Note that $Y_{\sigma(1)}\pitchfork Y_{\sigma(2)}\pitchfork Y_{\sigma(3)}$ implies $\iota(\sigma(2))$ and $\tau(\sigma(2))$ both exist. Hence by our hypothesis and the triangle inequality it suffices to show
	\[d_{Y_{\sigma(2)}}(Y_{\tau(\sigma(2))}, Y_{\sigma(3)})\le \bsi+2
	\quad\text{and}\quad
	d_{Y_{\sigma(2)}}(Y_{\sigma(1)}, Y_{\iota(\sigma(2))}) \le \bsi+2.\]
	We prove the first of these inequalities, the other being analogous.
	Set $a_1 = \sigma(2)$, and for $k\ge 1$ suppose we have constructed a sequence $\sigma(2) = a_1 < \dots <a_k \le \sigma(3)$ so that $a_{i+1}= \tau(a_i)$ for each $1 \le i < k$. If $Y_{a_k}\pitchfork Y_{\sigma(3)}$, then necessarily $\tau(a_k)\le \sigma(3)$ and we may set $a_{k+1} = \tau(a_k)$ to get a longer sequence $a_1 < \dots < a_{k+1}$ satisfying the same condition. We continue recursively in this manner until we obtain a maximal sequence for which $Y_{a_k}$ and $Y_{\sigma(3)}$ do not overlap. Thus $d_{Y_{\sigma(2)}}(Y_{a_k}, Y_{\sigma(3)}) \le 2$ by (\ref{eqn:multicure_bdd_projection}). Since $a_1 = \sigma(2)$ and $a_2 = \tau(\sigma(2))$, it now suffices to show $d_{Y_{a_1}}(Y_{a_2}, Y_{a_k})\le \bsi$.
	
	We claim the sequence $Y_{a_1}, \dots, Y_{a_k}$ satisfies the hypotheses of \Cref{cor:projections_persist}. Indeed, the fact $a_{j} = \tau(a_{j-1})$ ensures that $Y_{a_{j-1}}\pitchfork Y_{a_j}$ for each $1<j \le k$. This further ensures $a_{j-1} \le \iota(a_j)$. Hence, since $\tau(a_{j-1}) = a_{j} > \iota(a_{j})$ is the first index overlapping with $Y_{a_{j-1}}$, it must be that $Y_{a_{j-1}}$ and $Y_{\iota(a_{j})}$ do not overlap. Therefore their boundaries are disjoint and we conclude  $d_{Y_{a_{j}}} (Y_{a_{j-1}}, Y_{\iota(a_{j})}) \le 2$ by (\ref{eqn:multicure_bdd_projection}). Since $a_{j+1} = \tau(a_{j})$, when $j < k$, it now follows from the triangle inequality that
	\[d_{Y_{a_{j}}}(Y_{a_{j-1}}, Y_{a_{j+1}}) \ge d_{Y_{a_{j}}}(Y_{\iota(a_{j})}, Y_{\tau(a_{j})}) - 2 \ge \bgit + 3\bsi\]
	as required. Therefore \Cref{cor:projections_persist} applies to $Y_{a_1},\dots Y_{a_{k}}$. From this, we easily obtain the desired bound $d_{Y_{a_1}}(Y_{a_2}, Y_{a_k})\le \bsi$: Indeed, if $k = 2$ this is immediate, and if $k > 2$ the bound follows from \Cref{thm:behrstock} and the conclusion $d_{Y_{a_2}}(Y_{a_1}, Y_{k}) > \bsi$ of \Cref{cor:projections_persist}. This completes the proof when $m=3$.
	
	We next prove the general case by inducting on $m$. The cases $m\le 2$ are vacuous, so assume $m\ge 3$ and fix $i < j < k$. By induction the shorter sequences $Y_{\sigma(1)}, \dots, Y_{\sigma(j)}$ and $Y_{\sigma(j)},\dots Y_{\sigma(k)}$ satisfy the conclusion. Therefore \Cref{cor:projections_persist}(\ref{item:persistence}) (c.f.~\Cref{rem:general-to-special-persitence}) implies  $Y_{\sigma(i)}\pitchfork Y_{\sigma(j)}\pitchfork Y_{\sigma(k)}$. Hence the length 3 case implies the desired bound $d_{Y_{\sigma(j)}}(Y_{\sigma(i)},Y_{\sigma(k)}) \ge \bgit+3\bsi$.  
\end{proof}

\subsection{The mapping class group} \label{ssub:MCG_Reducibles}
Recall that $S$ is a connected orientable surface with empty boundary and finite complexity $\xi(S) \ge 1$.
We use $\Mod(S)$ to denote the mapping class group of a surface $S$, defined by 
\[\Mod(S)=\Homeo^+(S)/\Homeo_0(S),\] 
where $\Homeo^+(S)$ denotes the group of orientation preserving homeomorphisms of $S$ and $\Homeo_0(S)$ is the normal subgroup consisting of homeomorphisms isotopic to the identity map. 

Recall that an element $f\in \Mod(S)$
is \emph{periodic} if it has finite order in $\Mod(S)$, 
is \emph{reducible} if there exists a multicurve $\alpha$ so that $f(\alpha) = \alpha$,
and is \emph{pseudo-Anosov} if there is a number $\lambda > 1$ and a transverse pair of singular measured foliations $\mathcal{F}_{\pm}$ on $S$ so that $f(\mathcal{F}_\pm) = \lambda^{\pm1}\mathcal{F}_\pm$. The Nielsen--Thurston classification says that every element $f$ of $\Mod(S)$ is either periodic, infinite-order reducible, or pseudo-Anosov; see \cite[Chapter 13]{FarbMargalit-Primer}.
The prototypical example of a non-periodic reducible element is the \emph{Dehn twist} $T_\alpha$ about a curve $\alpha$, defined by cutting $S$ along $\alpha$ and then regluing with a full twist.

A mapping class $f\in \Mod(S)$ is said to be \emph{supported} on a subsurface $Y$ if it has a representative that restricts to the identity in the complement of $Y$. We moreover say $f$ is \emph{fully supported} on $Y$ if the restriction $f\vert_Y$ is a pseudo-Anosov element of $\Mod(Y)$ or if $Y$ is an annulus with $f\vert_Y$ nontrivial. In the former case we call $f$ a \emph{partial pseudo-Anosov} with support $Y$ and in the latter case a \emph{twist} about the core curve of the annulus. Thus any twist is simply a nontrivial power of a Dehn twist.

An element $g\in \Mod(S)$ is said to be \emph{pure} or in \emph{normal form} if it can be written as a product $g = f_1\dots f_k$ where each $f_i$ is fully supported on some subsurface $Y_i$ and these supporting subsurfaces have pairwise disjoint representatives (recall that in our formulation, subsurfaces are necessarily connected and non-empty).
In this case, a supporting subsurface $Y_i$ is called a \define{domain of $g$} 
if $f_i$ is a partial pseudo-Anosov on $Y_i$ or if $f_i$ is a twist and the annulus $Y_i$ is not homotopic into the support $Y_j$ for any partial pseudo-Anosov factor $f_j$.
 It is a fact that there is a uniform number $N\ge 1$ depending only on $S$ so that $f^N$ is in normal form for any element $f\in \Mod(S)$ \cite{Ivanov-SubgroupsTeichmullerModularGroups,BirmanLubotzkyMcCarthy}.

It is a crucial result of Masur and Minsky that pseudo-Anosov mapping classes act loxodromically on the curve graph, and in fact with a uniform lower bound on the asymptotic translation length \cite[Proposition 3.6]{MasurMinsky-CurveComplexI}. It is also easy to see that twists acts with translation length at least $1$ on the curve graph of the associated annulus. These facts lead the following consequence for pure mapping classes:

\begin{corollary}[{\cite[Corollary 2.11]{Mangahas-ShortWordpAs}}]
\label{thm:minimal_translation_mangahas_cor}
There exists a constant $c=c(S)>0$ such that for any pure element $g\in \Mod(S)$, any domain $Y$ of $g$, and any curve $\gamma\in\mathcal C(S)$ with nontrivial projection onto $Y$, we have 
\[d_Y(g^n\gamma, \gamma)\geq c\abs{n} \qquad\text{for all nontrivial $n\in\mathbb Z$}.\]
\end{corollary}

\subsection{Distance formula}
We shall need one more foundational result of Masur and  Minsky. 
A maximal multicurve has the property that every complementary component is a pair of pants, and therefore such multicurves are deemed \textit{pants decompositions}; any pants decomposition has $\xi(S)$ many components.  A \textit{marking} on $S$ is a pants decomposition $\mathcal{P} = \left\{\alpha_1,..., \alpha_{\xi(S)} \right\}$ together with a collection of \textit{transversal curves} $\left\{\mu_1,..., \mu_{\xi(S)} \right\}$ so that $i(\mu_i, \alpha_j) = 0$ whenever $i \neq j$ and so that $\mu_i$ intersects $\alpha_i$ the minimum number of times possible.

Markings $\mu$ have the key feature that for every subsurface $Y$, the projection $\pi_Y(\mu)$ is non-empty and has diameter at most $6$. Masur and Minsky showed that projections of markings to subsurfaces can be used to estimate distance in the mapping class group:

\begin{theorem}[Distance Formula \cite{MasurMinsky-CurveComplexII}]
\label{thm:DistanceFormula}
For any marking $\mu$ on $S$ and any finite generating set $X$ of $\Mod(S)$, there exists a constant $J_0\ge 1$ such that for each $J\ge J_0$ there exists $D\ge 1$ such that the word length of every element $f\in \Mod(S)$ can be estimated as
\[\abs{f}_X \asymp_D \sum_{Y\subset S} [[ d_Y(f\mu, \mu)]]_J \]
where $A\asymp_D B$ means $A \le DB +D$ and $B\le DA+D$, 
and where $[[x]]_{J}$ means $x$ whenever $x \ge J$ and means $0$ otherwise. 
\end{theorem}

\section{Reducible subgroups} 
\label{sec:reducible_subgroups}

In this section we gather the needed background concerning reducible subgroups of mapping class groups.

\begin{definition}
A subgroup $H\le \Mod(S)$ is \define{reducible} if there is a multicurve $\alpha$ so that $h(\alpha) = \alpha$ for all $h\in H$. Any such $\alpha$ is called a \emph{reducing multicurve} for $H$.
\end{definition}

The following classical theorem of Ivanov generalizes the Nielsen--Thurston classification from elements to subgroups:

\begin{theorem}[Ivanov \cite{Ivanov-SubgroupsTeichmullerModularGroups}]
\label{thm:Ivanov-reducible}
Every subgroup of $\Mod(S)$ is either finite, reducible, or contains a pseudo-Anosov element.
\end{theorem}

\begin{corollary}
An infinite subgroup $H$ of $\Mod(S)$ is virtually reducible if and only if it is reducible.
\end{corollary}

\begin{proof}
If $H$ is reducible it is clearly also virtually reducible. Conversely,
suppose $H$ has a finite-index reducible subgroup $H_0$. 
If $H$ fails to be reducible, then it contains a pseudo-Anosov element $f\in H$ by \Cref{thm:Ivanov-reducible}.  
Then
$f^{k!}\in H_0$ where $k=[H:H_0]$. But this contradicts the fact that $H_0$
is reducible. 
\end{proof}

It is well known that each reducible element $f$ has a \emph{canonical reducing system} $\partial f$. This can be characterized in multiple ways; we follow the approach of Handel--Thurston \cite[\S2]{HandelThurston}. Define:
\begin{itemize}
	\item $R(f)$ to be the set of all curves $\alpha$ whose orbit $\{f^k(\alpha) \mid k\in \Z\}$ is finite;
	\item $\partial f$ to be the set of elements of $R(f)$ that are disjoint from all other elements of $R(f)$.
\end{itemize}
This associates a (possibly empty) multicurve to each element $f\in \Mod(S)$ that is characterized by the property that $\{f^k(\beta)\mid k\in \Z\}$ is infinite for any curve $\beta$ intersecting $\partial f$. Note $f(\partial f) = \partial f$ by construction and that $\partial f$ is clearly empty whenever $f$ is periodic or pseudo-Anosov. In \cite[Lemma 2.2]{HandelThurston}, Handel and Thurston show $\partial f$ is nonempty whenever $f$ is reducible and infinite order. This can be extended to reducible subgroups in exactly the same way:

\begin{definition}
\label{def:canonical_reducing_sys}
The \define{canonical reducing system} $\partial H$ of a subgroup $H\le \Mod(S)$ consists of those elements of $R(H)$ that are disjoint from all other elements of $R(H)$, where $R(H)$ denotes the set of curves $\alpha$ whose orbit $H\cdot \alpha$ is finite.
\end{definition}
\begin{remark}
\label{rem:reducing_system_for_subgroups}
If $H'\le H$, then $R(H) \subset R(H')$ and therefore  $d_S(\partial H', \partial H) \leq 1$. If, moreover $[H':H]<\infty$, then $R(H) = R(H')$ and thus $\partial H' = \partial H$.
\end{remark}

While it is not obvious that $\partial H$ should be nonempty, the argument from \cite[Lemma 2.2]{HandelThurston} goes through with only minor adjustments to prove: 

\begin{lemma}
\label{lem:reducing_multicurve}
If $H\le \Mod(S)$ is reducible and infinite, then $\partial H$ is a nonempty, reducing multicurve for $H$. In contrast, $\partial H$ is empty whenever $H$ is finite or contains a pseudo-Anosov element.
\end{lemma}
\begin{proof}
If $H$ is finite, then $R(H)$ consists of all curves on $S$ and so $\partial H$ is empty. Similarly, if $H$ contains a pseudo-Anosov, then $R(H)$ is empty and so is $\partial H$. 

It remains to suppose $H$ is reducible and infinite. 
According to \cite[Lemma 2.1]{BKMM-GeomRigidityMCGs}, the set of (possibly disconnected) essential subsurfaces of $S$ forms a lattice under the relation of essential containment. In particular, any finite collection $\alpha_1,\dots,\alpha_n$ of curves (or, more formally, their annular neighborhoods) \define{fill} a well-defined subsurface $Y$ that is characterized as the unique topologically minimal subsurface essentially containing each curve $\alpha_i$.
 Following \cite{HandelThurston}, let $\mathcal{S}$ denote collection of all subsurfaces $Y$ that are filled by some finite subset of $R(H)$. The set $\mathcal{S}$ is partially ordered by inclusion.
Notice that every chain $Y_1 \subsetneq Y_2 \subsetneq\dots$ is finite, since the Euler characteristic must decrease at each step in a chain and all subsurfaces have Euler characteristic bounded below by that of the entire surface. Hence $\mathcal{S}$ has a maximal element $Y$, say filled by a finite list $\Gamma = \{\alpha_1,\dots,\alpha_n\}$ of curves in $R(H)$. Since each $\alpha_i$ has a finite $H$--orbit, it follows that the set of ordered tuples $\{(h\alpha_1,\dots, h \alpha_n)\mid h\in H\}$ is finite. The orbit stabilizer theorem thus implies there is a finite-index subgroup $H'\le H$ that fixes each curve $\alpha_1,\dots,\alpha_n$.

Suppose now that $Y = S$, meaning  $S$ is filled by $\Gamma = \{\alpha_1,\dots, \alpha_n\}$. By the Alexander trick, any element fixing each curve in $\Gamma$ is isotopic to the identity. Thus $H'$ is the trivial group and  $|H|< \infty$. As this contradicts our hypothesis, it must be that $Y \neq S$ is a proper subsurface. 

Notice that the subsurface $Y$ has finite $H$--orbit and thus its boundary $\partial Y$ is contained in $R(H)$. If $\partial Y$ is not in $\partial H$, then there must be some curve $\gamma$ intersecting it with finite $H$-orbit. In this case $\Gamma' = \{\alpha_1,\dots,\alpha_n,\gamma\}\subset R(H)$ fills a strictly larger surface $Y'$, contradicting the maximality of $Y$ in $\mathcal{S}$. Therefore $\partial H$ contains $\partial Y$ and is non-empty.
 \end{proof}

Note that the above proof actually characterizes $\partial H$ as the union of $\partial Y$ over all maximal subsurfaces $Y$ in the partial order. Indeed, this union is clearly contained in $\partial H$; conversely, given some $\alpha \in \partial H$,  if $\alpha$ is not on the boundary of any such maximal subsurface, it must live in the interior of one. This implies that there are curves intersecting $\alpha$ with finite $H$-orbit, contradicting the definition of $\partial H$.

\begin{example}\label{remark:reducibleExamples}
Let us illustrate some ways to construct reducible subgroups and describe their canonical reducing systems.
\begin{itemize}

 \item For any multicurve $\alpha \subset S$ its stabilizer $\stab(\alpha) = \{h\in \Mod(S) : h\alpha =\alpha\}$ is a reducible subgroup whose boundary is precisely $\alpha$. Every reducible subgroup is naturally a subgroup of some multicurve stabilizer, namely that of its boundary.
\item Suppose $Y=Y_1 \sqcup \cdots \sqcup Y_n$ is a collection of disjoint, essential subsurfaces and consider infinite subgroups $H_i\leq \Mod(Y_i)\leq \Mod(S)$ (See \Cref{fig:multisurface}). Let $H$ be a reducible group 
generated by $H_1,\ldots,H_n$ and any collection of elements of $\Mod(S)$ preserving $Y$.
In general, we have $d_S(\partial Y, \partial H)\leq 1$. 
If additionally each $H_i$ contains a fully supported element, 
it follows that $\partial H =\partial Y$.

 \end{itemize}
\begin{figure}[h]
	\includegraphics[width=.7\textwidth]{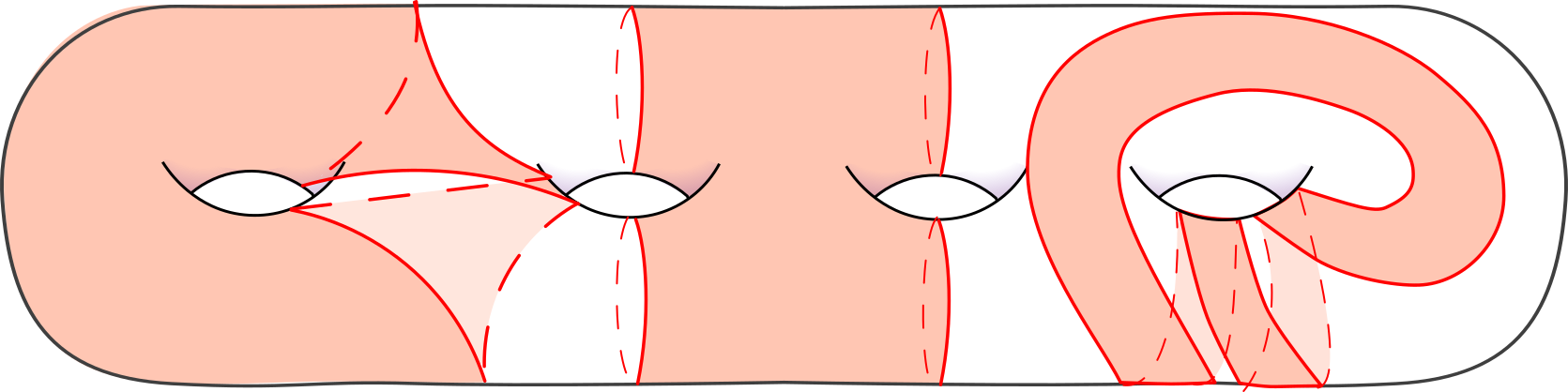}
	\caption{
		A disjoint union $Y= Y_1 \sqcup Y_2 \sqcup Y_3$ of essential subsurfaces. }
		\label{fig:multisurface}
\end{figure}
\end{example}

\subsection{Multitwist groups}\label{subsec:multitwist}
We shall also be interested in the following restricted class of reducible subgroups, which we will call multitwist groups.

A \emph{multitwist} in $\Mod(S)$ is any element $f = T_{\alpha_1}\dotsb T_{\alpha_k}$ that can be expressed as a product of commuting Dehn twist $T_{\alpha_i}$; that is, any element supported on a disjoint union of annuli. We call a subgroup $H\le \Mod(S)$ a \define{multitwist group} if every element of $H$ is a multitwist. 
We note the following simple observation:

\begin{lemma}
\label{lem:multitwist}
A subgroup $H$ is a multitwist group if and only if there is a multicurve $\alpha = (\alpha_1,\dots,\alpha_k)$ so that $H$ is contained in $\langle T_{\alpha_1},\dots,T_{\alpha_k}\rangle\cong \Z^k$. In particular, $H$ is abelian and reducible.
\end{lemma}
\begin{proof}
Let $C$ be the set of all curves $\alpha_i$ appearing (with nontrivial power) in any elements $f = T_{\alpha_1}^{k_1},\dots, T_{\alpha_m}^{k_m}$ of $H$. We claim the elements of $C$ are pairwise disjoint. This will prove the claim since then $C$ is a multicurve and $H$ is contained in the group generated by the set $\{T_\alpha\}_{\alpha\in C}$.

If the claim is false, we can find two elements $f = T_{\alpha_1}^{k_1}\dotsb T_{\alpha_m}^{k_m}$ and $g = T_{\beta_1}^{\ell_1}\dotsb T_{\beta_n}^{\ell_n}$ so that some $\alpha_i$ intersects some $\beta_j$. In this case, let $Y \subset S$ be the subsurface filled by $\alpha_{i}$ and $\beta_{j}$. Penner's generalization of Thurston's construction of pseudo-Anosov homeomorphisms \cite{Penner} implies that $T^{k}_{\alpha_{i}} T^{-l}_{\alpha_{j}}$ is pseudo-Anosov on $Y$ for any positive integers $k,l$. It follows that the normal form of $fg^{-1}$ contains a partial pseudo-Anosov factor, and in particular, it is not a multi-twist. \end{proof}

We say a subgroup $H\le \Mod(S)$ is virtually a multitwist group if it has a finite index subgroup $H_0\le H$ that is a multitwist group. \Cref{lem:multitwist} shows that multitwist subgroups are exactly the class of groups considered by Loa in \cite{Loa}.

\section{Geometrically finite subgroups of the mapping class group} 
\label{ssub:RGF}

In the classical setting of Kleinian groups, geometric finiteness can be viewed as a \emph{relative} version of convex cocompactness that allows for parabolic isometries in certain prescribed subgroups. 
As an illustrative example, consider a 
complete, finite-volume, cusped hyperbolic $3$--manifold $M$, which has the feature that all elements of $\pi_1(M)\le \mathrm{Isom}(\mathbb{H}^3)$ are loxodromic aside from those that are conjugate into the parabolic $\Z^2$ subgroups corresponding to the toroidal cusps of $M$.

Motivated by this analogy, Dowdall, Durham, Leininger and Sisto defined in \cite{DowdallDurhamLeiningerSisto-ExtensionsII} the notion of parabolically geometrically finite subgroups of $\Mod(S)$ to capture the idea of being relatively convex cocompact in a way that is compatible with the presence for multitwist elements, which are precisely the parabolic isometries of Teichm\"uller space. Udall \cite[Definition 6.4]{Udall-combinations_of_PGF} later expanded this definition to allow for peripheral subgroups containing more general reducible elements. This leads to the following formulation:

\begin{definition}
\label{def:RGF}
We say a subgroup $G<\Mod(S)$ is {\em reducibly geometrically finite
(RGF) relative to a collection $\mathcal H=\{H_1,\ldots,H_n\}$} of reducible subgroups of $G$ if
\begin{enumerate}
\item $G$ is hyperbolic relative to the collection $\mathcal H$, and 
\item the coned off Cayley graph $\widehat{\Gamma}(G;\mathcal{H})$ of $G$ with respect to $\mathcal{H}$ admits a $G$--equivariant coarse quasi-isometric embedding into the curve graph $\cc(S)$.
\end{enumerate}
Such an subgroup is more specifically {\em parabolically geometrically
finite (PGF)} relative to $\mathcal{H}$ if each subgroup $H_i$ is virtually a multitwist group. 
We also  say that $G$ is \define{RGF}/\define{PGF} if it is so relative to some finite collection $\mathcal H$ of subgroups.
\end{definition}

\begin{remark}
We note that this differs slightly from Udall's formulation \cite[Definition 6.4]{Udall-combinations_of_PGF}, which additionally requires the subgroups $H_i$ to be ``virtually pure reducible strongly undistorted;'' an assumption which allows one to prove \cite[Theorem 6.8]{Udall-combinations_of_PGF} the subgroup $G$ is undistorted in the mapping class group.
\end{remark}

\subsection{Known examples of geometrically finite subgroups}
A main goal of this paper is to present new constructions of RGF and PGF subgroups of the mapping class group. To give context, here we survey the landscape of geometric finiteness and review the known examples from the literature.

The notion of parabolic geometric finiteness was introduced in \cite{DowdallDurhamLeiningerSisto-ExtensionsII} in the context of studying the geometry of surface group extensions associated to lattice Veech groups. Recall that to each subgroup $G\le \Mod(S)$ there is an associated $\pi_1(S)$--extension group $\Gamma_G$ obtained by taking the preimage of $G$ under the forgetful map $\Mod(S,p)\to \Mod(S)$ of the Birman exact sequence. Recall also that a subgroup $G\le \Mod(S)$ is a \define{Veech group} if it stabilizes a Teichm\"uller disk $D$, and that it is moreover a \define{lattice} if the quotient $D/G$ has finite volume. In \cite{DowdallDurhamLeiningerSisto-ExtensionsII} it was shown that if $G$ is a lattice Veech group, then the associated extension $\Gamma_G$ is a hierarchically hyperbolic group. This result was later extended by Bongiovanni \cite{bongiovanni2024extensionsfinitelygeneratedveech} to handle all finitely generated Veech groups. 
Earlier work of Tang \cite{Tang} had moreover shown that finitely generated Veech groups satisfy the conditions to be PGF. These results give evidence that finitely generated Veech groups should qualify as ``geometrically finite'' and that \Cref{def:RGF} is a reasonable formulation of the notion.

In the spirit of the Klein--Maskit combination theorem for Kleinian groups, Leininger and Reid \cite{LR} gave a combination theorem for Veech subgroups which shows, in the simplest case, that if $G\le \Mod(S)$ is a Veech subgroup with a maximal parabolic subgroup $H\le G$, then for every ``sufficiently complicated'' partial pseudo-Anosov centralizing $H$, the amalgamated free product $G \ast_H \phi G\phi\inv$ embeds into $\Mod(S)$. Such combinations are interesting in part because they allow one to construct higher-genus surface subgroups of $\Mod(S)$ with the property that all elements are pseudo-Anosov except for a single conjugacy class. Udall \cite{Udall-combinations_of_PGF} has recently analyzed these Leininger--Reid combinations from the new perspective of geometric finiteness and shown they are indeed PGF. In fact, Udall proves a general combination theorem, showing that an amalgamated free product of PGF groups over parabolic subgroups will both embed into $\Mod(S)$ and be PGF, provided a technical ``$L$--local large projections'' property is satisfied (analogous to the above ``sufficiently complicated'' assumption on the partial pseudo-Anosov $\phi$).

Finally, as indicated in \Cref{intro}, Loa \cite{Loa} considered free products of two multitwist subgroups $H_1,H_2\le \Mod(S)$ and proved there is a constant $D = D(S)$ such that if $d_S(\partial H_1, \partial H_2)\ge D$ then the free product $H_1\ast H_2$ is PGF and embeds in $\Mod(S)$. This work was the motivation of our \Cref{main-Loa}.

\section{Bass--Serre trees and free products}
\label{s:cayley_trees}

Let $\mathcal{H} = \left\{H_1, \ldots, H_n\right\}$ be a family of nontrivial groups. In this section we will 
describe the Bass--Serre tree $T = T_{\mathcal{H}}$ associated to the abstract free product $H\colonequals H_1\ast \dots \ast H_n$; see \cite{ScottWall-TopologicalMethodsGrpThry}
for a general overview of Bass--Serre tree theory.

A natural graph of groups decomposition of the free product is given by the star graph $T_0$ (that is, the complete bipartite graph $K_{1,n}$) in which the internal vertex and the edges are all labeled by the trivial group, and the $n$ nodes are labeled by the given groups $H_1,\dots,H_n$. The associated Bass--Serre tree $T$ comes equipped with an action $H\curvearrowright T$ of the free product and an associated quotient map $T\to T_0$. 

The bipartite structure of $T_0$ lifts to a bipartite structure on $T$ in which we call lifts of the node vertices \define{type--1} and lifts of the central vertex \define{type--2}. Each type--2 vertex has trivial stabilizer, and each type--1 vertex has stabilizer equal to a conjugate $gH_i g\inv$ of the corresponding node vertex.
By the orbit-stabilizer theorem, the type--2 vertices are in bijective correspondence with $H$ itself and the type--1 vertices mapping to the $H_i$ node in $T_0$ are in correspondence with the set $H/H_i$ of cosets of $H_i$. Accordingly, we may label the vertices of $T$ as follows: 
\begin{itemize}
	\item The {\em type}--$1$ vertex with stabilizer $gH_ig\inv$ is labeled by the coset $gH_i$.
	\item The {\em type}--$2$ vertices are vertices labeled by group elements $g\in H$.
\end{itemize}
We write ${\bf v}(\ast)$ for the vertex of $T$ with label $\ast$. 
The vertex ${\bf v}(gH_i)$ thus has valence equal to $\abs{H_i}$ and is connected precisely to the vertices labeled by the elements of the coset $gH_i$. 
Correspondingly, the vertex ${\bf v}(g)$ has valence $n$ and is connected precisely to the vertices ${\bf v}(gH_1),\dots,{\bf v}(g H_n)$; see \Cref{fig:BassSerreTree} below. 

\begin{figure}[h]
	\centering

	\includegraphics[height=7cm]{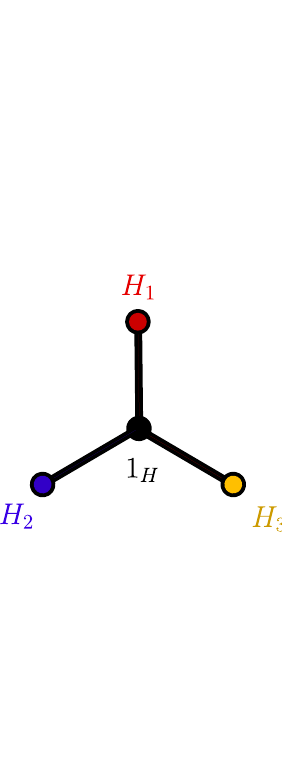}
	\hspace{2em}
	\includegraphics[height=7cm]{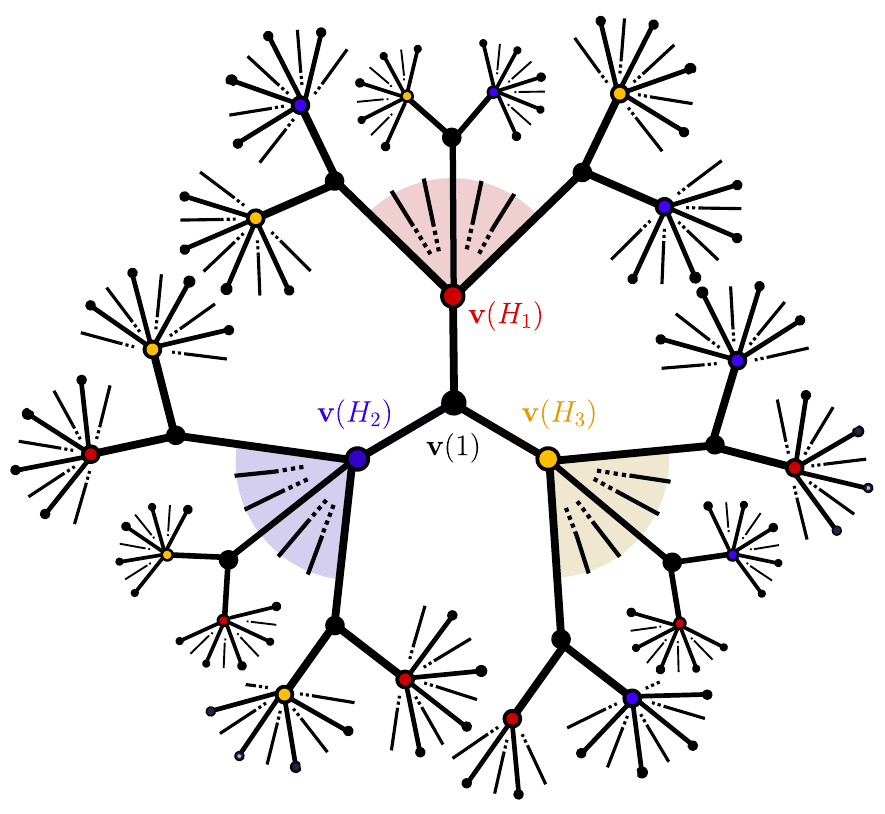}

	\caption{
	A graph of groups decomposition for $H = H_1\ast H_2 \ast H_3$. On the left, we see the star graph $T_0$; on the right, the Bass--Serre $T$ associated to the free product. 
	Each black type--1 vertex corresponds to an element of the group with the identity element ${\bf v}(1)$ labeled in the center. Each red (resp. blue, yellow) type--2 vertex corresponds to a coset of the form ${\bf v}(gH_1)$ (resp. ${\bf v}(gH_2)$, ${\bf v}(gH_3)$). 
}
	\label{fig:BassSerreTree}
	\label{fig:GraphOfGroups}
\end{figure}

Observe that the action of the free product $H = H_1\ast \dots \ast H_n$ of the tree $T$ satisfies the \Cref{def:rel_hyp} of relative hyperbolicity. In particular, by \Cref{rem:bowditch_tree_qi_coned-off_Cayley} it is equivariantly quasi-isometric to the coned-off Cayley graph.

\subsection{Free products in mapping class groups}
\label{sec:free_products_in_MCG}
Now let us suppose that our groups $H_1,\dots, H_n$ are in fact infinite, reducible subgroups of the mapping class group. That is, for each $1\leq i\leq n$ we have an inclusion $H_i\to \Mod(S)$. By the universal property of the free product, these determine a morphism
\[\Phi\colon H\to \Mod(S)\]
whose image is the subgroup $G = \langle H_1,\dots,H_n\rangle\le\Mod(S)$ they generate.

The homomorphism $\Phi$ induces an action of $H$ on the set of curves, multicurves, and subsurfaces of $S$; namely, if $\alpha$ is a (multi)curve or subsurface of $S$, then we write 
$g\cdot\alpha = \Phi(g)\alpha$. 
Note that when restricted to $\cc(S)$, this $H$--action is by isometries.
This action allows us to define an $H$--equivariant map
\[\phi\colon T\to \cc(S)\]
as follows: Send the type--1 vertex ${\bf v}(g H_i)$ to the reducing system $\partial\Phi(gH_ig\inv) = g\cdot\partial H_i$ of the image of the corresponding stabilizer subgroup $g H_i g\inv$. 
Then fix a curve $\xi$ in $\cc(S)$ to be the image ${\bf v}(1)$ and, by equivariance, define $\phi$ to send the type--2 vertex ${\bf v}(g)$ to the curve $g\cdot \xi$. Note that $\phi$ is only a coarse map since the image of a type--1 vertex ${\bf v}(gH_i)$ is an entire multicurve $g\cdot \partial H_i$ of diameter at most 1: this is necessary since the stabilizer $gH_ig\inv$ must fix the image but may permute the components of $g\cdot\partial H_i$.

Since $T$ is $H$--equivariantly quasi-isometric to the coned-off Cayley graph of $H$, in order to prove $G = \langle H_1,\dots, H_n\rangle$ is an RGF subgroup of $\Mod(S)$, it suffices to show $G$ is isomorphic to $H$ and that our map $\phi\colon T\to \cc(S)$ is a quasi-isometric embedding. We record this in the following lemma:

\begin{lemma}
\label{lem:qi-to-CS-implies_RGF}
Let $\mathcal{H} = \{H_1,\dots, H_n\}$ be a family of infinite reducible subgroups $H_i\le \Mod(S)$. Let $G = \langle H_1,\dots, H_n\rangle\le \Mod(S)$ be the subgroup they generate, $T$ the Bass--Serre tree for the abstract free product $H = H_1\ast \dots \ast H_n$, and $\phi\colon T\to \cc(S)$ the map defined above. If there exists a constant $\kappa >0$ such that
\[d_S(\phi(v), \phi(v')) \ge \tfrac{1}{\kappa} d_T(v, v')-\kappa\]
for all type--1 vertices $v,v'$ of $T$, then $\Phi\colon H\to G$ is an isomorphism and $G$ is reducibly geometrically finite relative to  the collection $\mathcal{H}$. Furthermore, every element of $G$ which is not conjugate into some $H_i$ is pseudo-Anosov.
\end{lemma}
\begin{proof}
The lemma is clear when $n=1$, since then $H = H_1 = G$ is trivially RGF relative to itself. So we assume $n\ge 2$.

The map $\Phi$ is surjective by construction. To see it is injective, let $h\in H$ be any nontrivial element and consider its minset \[A =\left\{x\in T: d_T(x,h\cdot x) = \inf_{y\in T}d_T(y,h\cdot y)\right\}.\] There are two possibilities: either $h$ acts elliptically and $A$ consists of a single type--1 vertex (namely ${\bf v}(g H_i)$  iff $h\in g H_i g\inv$), or else $h$ acts loxodromically and $A$ consists of its bi-infinite translation axis. In either case, it is known that
for every $v\in T$ the geodesic $[v, h\cdot v]$ must pass through the minset $A$.  
Hence, by choosing a type--1 vertex $v$ sufficiently far from $A$ we may be assured that $d_T(v, h\cdot v)  > \kappa(\kappa+1)$. The hypothesis of the lemma then ensures
\[d_S(\phi(v), \Phi(h)\phi(v)) = d_S(\phi(v), \phi(h\cdot v))  \ge \tfrac{1}{\kappa} d_T(v, h\cdot v) -\kappa > 1.\]
By construction $\phi(v)$ is a reducing multicurve (associated to some reducible subgroup $\Phi(gH_ig\inv)$) of diameter at most $1$. Therefore the above inequality implies $\Phi(h)\phi(v)\ne \phi(v)$. In particular $\Phi(h)\ne 1$, proving that $\Phi$ is injective.

This proves $G$ is a free product and so hyperbolic relative to the subgroups $H_1,\dots, H_n$. By \Cref{rem:bowditch_tree_qi_coned-off_Cayley}, the Bass--Serre tree $T$ is thus quasi-isometric to the coned off Cayley Graph $\widehat\Gamma = \widehat\Gamma(G, \mathcal{H})$ defined in \S\ref{sec:rel_hyp}. Indeed, there is a bijection between the vertices of $\widehat{\Gamma}$ and $T$ that sends each coset vertex $gH_i$ of $\widehat{\Gamma}$ to the corresponding type--1 vertex ${\bf v}(gH_i)$ of $T$ and each regular vertex $g\in \Gamma$ to the type--2 vertex ${\bf v}(g)$ of $T$. This map is bi-Lipschitz, since each edge of $T$ maps to an edge of $\widehat{\Gamma}$ and each edge of $\widehat{\Gamma}$ either maps to an edge of $T$ (when it involves a coset vertex) or an path of bounded length $d_T({\bf v}(1), {\bf v}(x)) = d_T({\bf v}(g),{\bf v}(gx))$ when it is an edge $(g,gx)$ of $\Gamma$ labeled by an element $x$ of the finite generating set of $G$.

To prove $G$ is RGF it remains to show $\phi\colon T\to \cc(S)$ is a quasi-isometric embedding, which is easy: any adjacent vertices of $T$ have the form ${\bf v}(gH_i)$ and ${\bf v}(gh)$, for some $h\in H_i$, and are mapped to the subsets $g\cdot\partial H_i$ and $gh\cdot \xi$ of $\cc(S)$ at distance
\[d_S\left(g\cdot\partial H_i, gh\cdot\xi\right) = d_S(h\inv\partial  H_i, \xi) = d_S(\partial H_i, \xi). \]
Therefore $\phi$ is $\lambda$--Lipschitz for $\lambda = \max_{j} d_S(\partial H_j, \xi)$. Conversely, given any vertices $w_1,w_2$ of $T$ we may find type--1 vertices $v_1,v_2$ with $d_T(v_i,w_i)\le 1$ and conclude
\begin{align*}
d_S(\phi(w_1),\phi(w_2)) &\ge d_S(\phi(v_1),\phi(v_2)) -2\lambda\\
 &\ge \tfrac{1}{\kappa} d_T(v_1,v_2)-2\lambda 
\ge \tfrac{1}{\kappa} d_T(w_1,w_2) - 2(\tfrac{1}{\kappa}+\lambda).
\end{align*}

Finally, if $g\in G$ is not conjugate into some $H_i$, then $\Phi\inv(g)\in H$ does not fix any vertex of $T$ and so acts as a loxodromic isometry. 
Thus $d_T((\Phi\inv(g))^n\cdot v, v)\to \infty$ for each type--1 vertex of $T$. Consequently $d_S(g^n \phi(v), \phi(v))\to \infty$ as well, which means $g$ is pseudo-Anosov.
\end{proof}

\section{Displacing families and the proof of \Cref{main-subgroups}}\label{s:displacing}

The following property will be key to quasi-isometrically embedding the coned off Cayley graph for a family $\mathcal{H}$ of reducible subgroups.

\begin{definition}[Displacing]
\label{def:displacing}
A family $\mathcal{H}=\{H_1, \dots, H_n\}$ of infinite reducible subgroups is \textbf{$L$--displacing} if there are multicurves $\beta_1,\dots,\beta_n$ such that for all  $i\ne j \ne k$:
\begin{enumerate}
\item $H_j$ stabilizes $\beta_j$, that is $h \beta_j = \beta_j$ for all $h\in H_j$;
\item $d_S(\beta_i,\beta_j)\ge 5$; and 
\item for each nontrivial $h\in H_j$, there exists a subsurface $Y$ with $d_S(\partial Y,\beta_j)\le 1$ such that $d_Y(\beta_i, h \beta_k)\ge L$.
\end{enumerate}
\end{definition}
\begin{remark}
\label{rem:separated_implicit_in_displacing}
Note that this definition allows for $i=k$. 
\end{remark}

As a concrete example to illustrate this property, suppose each $H_i$ is generated by a collection of fully supported mapping classes on disjoint subsurfaces as in \Cref{remark:reducibleExamples}. If the translation length of each mapping class on its supporting domain are uniformly bounded below by the $L$ threshold (with negligible additive constants), and the collection is at least $5$--separated, then $\{H_i\}$ is $L$--displacing. In this case, we use $\beta_i = \partial H_i$ and for $h\in H_j$ the $Y$ subsurface can be taken as any of the supporting domains for the mapping class in its normal form. 

\begin{remark}
While it is often natural to take $\beta_i = \partial H_i$ in the definition of displacing, it is sometimes advantageous to be more flexible. This is the case in \Cref{prop:finite-index-displacing-family} below, where we start with a $5$--separated family $\{G_i\}$ and pass to infinite index subgroups $H_i\le G_i$. In this case $d_S(\partial G_i, \partial H_i)\le 1$, so the new family $\{H_i\}$ need only be $3$--separated, but taking the original reducing systems $\beta_i = \partial G_i$ for our multicurves allows the family $\{H_i\}$ to be displacing.
\end{remark}

The significance of the $L$-displacing property is captured by the following
theorem, which is the main result of this section.

\begin{theorem}\label{thm:displacing}
If $\Hc=\{H_1,\dots, H_n\}$ is a $(\bgit+4\bsi)$--displacing family of infinite reducible subgroups (where $\bgit$ is from \Cref{thm:BddGeodesicImage} and $\bsi$ from \Cref{thm:behrstock}), then $G = \langle H_1, \dots, H_n \rangle$ is isomorphic to $H_1\ast\dots\ast H_n$ and is RGF relative to $\mathcal{H}$. Further, every element of $G$ which is not conjugate into a factor $H_i$ is pseudo-Anosov.
\end{theorem}
\begin{proof}
Let $\beta_1,\dots,\beta_n$ be the multicurves promised in the \Cref{def:displacing} of displacing. Note that since $H_i$ stabilizes $\beta_i$ we necessarily have $\beta_i\subset R(H_i)$ and thus $d_S(\beta_i,\partial H_i)\le 1$.
As in \Cref{s:cayley_trees}, let $H = H_1\ast\dots \ast H_n$ be the abstract free product, $T$ the Bass--Serre tree for $H$, and $\phi\colon T\to \cc(S)$ the $H$--equivariant map. The theorem will follow directly from \Cref{lem:qi-to-CS-implies_RGF} provided we can find a constant $\kappa > 0$ so that for all type--1 vertices $v,v'$ of $T$ we have
\begin{equation}
\label{eqn:lower-bound-displacing-thm}
d_S(\phi(v), \phi(v')) \ge \frac{1}{\kappa} d_T(v,v') - \kappa.
\end{equation}

\begin{figure}[h]
\begin{tikzpicture}[scale=.85]
	\def\s{0.3}
	\coordinate[label=left:{\footnotesize$a_0$}, circle, fill=black, scale=\s] (0) at (-4.5,0);
	\coordinate[label=above:{\footnotesize$b_0 = \textbf{v}(g_0)$}, circle, fill=black, scale=\s] (1) at (-3.5,.5);
	\coordinate[label=below:{\footnotesize$a_{s-1}=\textbf{v}(g_{s-1}H_i)$}, circle, fill=black, scale=\s] (2) at (-2,-.5);
	\coordinate[label=above:{\footnotesize$b_{s-1}=\textbf{v}(g_{s-1})$}, circle, fill=black, scale=\s] (3) at (0,.5);
	\coordinate[label={[xshift=-5pt] below :\footnotesize $\textbf{v}(g_{s-1}H_j)=a_s =\textbf{v}(g_{s} H_j)$}, circle, fill=black, scale=\s] (4) at (2,0);
	\coordinate[label=above:{\footnotesize $b_{s}=\textbf{v}(g_s)$}, circle, fill=black, scale=\s] (5) at (4,.5);
	\coordinate[label={[xshift=10pt] below:\footnotesize$a_{s+1}=\textbf{v}(g_s H_k)$}, circle, fill=black, scale=\s] (6) at (6,-.5);
	\coordinate[label={[xshift=10pt] above:\footnotesize$b_{r-1}=\mathbf{v}(g_{r-1})$}, circle, fill=black, scale=\s] (7) at (7,.5) ; 
	\coordinate[label={right:\footnotesize$a_r$}, circle, fill=black, scale=\s] (8) at (8,0) ; 
	\draw[black] (0) -- (1);
	\draw[dotted] (1) -- (2);
	\draw[black] (2) -- (3);
	\draw[black] (3) -- (4);
	\draw[black] (4) -- (5);
	\draw[black] (5) -- (6);
	\draw[dotted] (6) -- (7);
        \draw[black] (7) -- (8);
	
\end{tikzpicture}
\caption{The notation for a path in the Bass--Serre tree from the proof of \Cref{thm:displacing}.}
\label{fig:paths_in_thm_displacing}
\end{figure}

To that end, choose any type--1 vertices $v,v'\in T$ and consider the geodesic between them: 
If $d_T(v,v')=2r$ this is an alternating sequence $v = a_0,b_0, a_1, b_1, \dots,b_{r-1},a_r =v'$ of type--1 vertices $a_s$ and type--2 vertices $b_s$.
Each type--2 vertex $b_s$ has the form $b_s = {\bf v}(g_s)$ for some unique element $g_s\in H$.
Similarly, 
each type--1 vertex $a_s$ corresponds to a unique coset $a_s= {\bf v}(g H_j)$ for some index $j\in \{1,\dots, n\}$. Accordingly, let us define $\beta_s' = g\cdot\beta_j$. 
This is well-defined since if $g H_j = g' H_j$ then $g\inv g'\in H_j$ stabilizes $\beta_j$ and so $g\cdot \beta_j = g(g\inv g')\cdot\beta_j = g'\cdot\beta_j$. By the definition of $\phi$ we also note that 
\begin{equation}
\label{eqn:displacing0}
d_S(\beta_s', \phi(a_s)) = d_S(g\cdot \beta_j, g\cdot \partial H_j) \le 1.
\end{equation}

Now fix $0 < s < r$ and observe that the cosets labeling type--1 vertices $a_{s-1},a_s,a_{s+1}$ evidently contain the elements $g_{s-1}$ or $g_{s}$ (see \Cref{fig:paths_in_thm_displacing}); thus there are indices $i\ne j \ne k$ in $\{1,\dots, n\}$ so that $a_{s-1} = {\bf v}(g_{s-1}H_i)$ and $a_s = {\bf v}(g_{s-1}H_j) = {\bf v}(g_sH_j)$ and $a_{s+1} = {\bf v}(g_sH_k)$. By the preceding paragraph, we see that $\beta_s' = g_s\cdot \beta_j$ and $\beta_{s+1}' = g_s \cdot\beta_k$. 
Hence the definition of displacing implies:
\begin{equation}
\label{eqn:displacing1}
d_S(\beta_s', \beta_{s+1}') = d_S(g_s\cdot \beta_j, g_s\cdot \beta_k) \ge 5.
\end{equation}

Note also that $g_{s-1}\inv g_s\in H_j$ since $g_{s-1} H_j = g_s H_j$. Since $\mathcal{H}$ is displacing, we are thus provided a subsurface $Y$ with $d_S(\partial Y, \beta_j)\le 1$ so that 
\begin{equation}
\label{eqn:displacing2}
\bgit+4\bsi \le d_Y(\beta_i, (g\inv_{s-1} g_{s})\cdot \beta_k) = d_{g_{s-1}\cdot Y} (g_{s-1}\cdot \beta_i, g_{s}\cdot\beta_k) = d_{Y_s} (\beta_{s-1}', \beta_{s+1}'),
\end{equation}
where here and henceforth we write $Y_s = g_{s-1}\cdot Y$. 
To round out the notation, let $Y_0$, $Y_r$ denote any annulus corresponding to a component of $\beta_0'$ and $\beta_r'$
respectively.
In this way, we obtain a sequence of subsurfaces $Y_0,\dots, Y_r$ and multicurves $\beta_0',\dots, \beta_r'$ with $d_S(\partial Y_s, \beta_s')\le 1$ for each $s$.  It now follows from \Cref{eqn:displacing1} that $d_S(Y_s, Y_{s+1})\ge 3$. Moreover, \Cref{eqn:multicure_bdd_projection}
 and \Cref{eqn:displacing2} imply
\[d_{Y_s}(Y_{s-1}, Y_{s+1}) \ge d_{Y_s}(\beta_{s-1}',\beta_{s+1}') - 4 \ge \bgit+3\bsi.\]
Therefore the subsurfaces $Y_0,\dots, Y_r$ satisfy \Cref{cor:projections_persist} and we may conclude $d_S(Y_0, Y_r) \ge r$ and thus $d_S(\beta_0',\beta_r')\geq r-2$. 
Therefore, applying \Cref{eqn:displacing0}
establishes the lower bound required in \Cref{eqn:lower-bound-displacing-thm}:
\[ 
d_S(\phi(v),\phi(v')) = d_S(\phi(a_0),\phi(a_r)) \ge d_S(\beta_0',\beta_r') - 2 \ge r-4 = \tfrac{1}{2}d_T(v,v') - 4.\qedhere
\]
\end{proof}

\subsection{Passing to finite index}

To apply \Cref{thm:displacing}, it is natural to look for conditions that imply that a collection $\Hc$ of reducible subgroups is displacing, or to have a ready source of examples. The following lemma is the key tool we will use to construct such examples.

\begin{lemma}
\label{lem:displacing_for_one_group}
Let $G\le \Mod(S)$ be an infinite reducible subgroup. For any $L> 0$ and marking $\mu$, there exists a finite-index subgroup $G' \le G$ such that for each nontrivial $g\in G'$ there is a subsurface $Y$ which is disjoint from $\partial G' = \partial G$ for which
\[d_Y(g\mu, \mu)\geq L.\]
\end{lemma} 

\begin{proof}
Let $\mu'$ be a marking on $S$ containing the reducing system $\partial G$. It follows from the Distance Formula \Cref{thm:DistanceFormula} (or, more accurately, the marking-complex version \cite[Theorem 6.12]{MasurMinsky-CurveComplexII}) that there is a bound $K$ (depending on $\mu$ and $\mu'$) so that $d_W(\mu,\partial G) \le d_W(\mu,\mu')\le K$ for every subsurface $W$ of $S$.

Fix a generating set $X$ on $\Mod(S)$. Let $J_0$ be the constant from \Cref{thm:DistanceFormula} and apply that theorem with threshold $J = L + 3K + J_0$ and marking $\mu$ to obtain a quasi-isometry constant $D\ge 1$ for the Distance Formula. Since $\Mod(S)$ is residually finite \cite{Grossman-ResidualFinitenessMCG}, there is a finite-index subgroup $\Gamma\le \Mod(S)$ such that each nontrivial $f\in \Gamma$ has word length $\abs{f}_X \ge 2D$.

We claim that the finite-index subgroup $G' = \Gamma\cap G$ satisfies the conclusion. Indeed, if $g\in G'$ then $\abs{g}_X > D$ and hence the distance formula implies there must be some subsurface $Y$ with $d_Y(g\mu,\mu)\ge J \ge L+3K$ (for else the right hand side of the formula would be zero, and $\abs{g}_X\le D$). It must also be that $Y$ is disjoint from $\partial G$, since otherwise $\pi_Y(\partial G)\ne \emptyset$ and the triangle inequality would yield this absurdity:
\begin{align*} 
3K <  d_Y(g\mu, \mu)&\leq d_Y(g\mu, \partial G)+d_Y(\partial G, \mu)\\
                   &=d_Y(g\mu, g\partial G)+d_Y(\partial G, \mu) \quad\text{   (because $\partial G$ is fixed by $g\in G$)}\\
                   &=d_{g^{-1} Y}(\mu, \partial G)+d_{Y}(\partial G, \mu)\leq 2K.
\end{align*}
Finally, since $G'$ is finite-index in $G$, 
applying \Cref{rem:reducing_system_for_subgroups} gives $\partial G' = \partial G$.
\end{proof}

With this lemma in hand it is straight forward to prove the following proposition, which says the $L$--displacing property can always be achieved by passing to finite index subgroups, provided the original family is sufficiently separated:

\begin{definition}[Separated]
\label{def:separated} 
A family $\Hc = \{G_1,\dots,G_n\}$  of infinite, reducible subgroups $G_i \le \Mod(S)$ is \textbf{$D$--separated} if $d_S(\partial G_i , \partial G_j)\ge D$ for all $i\ne j$.
\end{definition}

\begin{proposition}
\label{prop:finite-index-displacing-family}
Let $\Gc = \{G_1, \dots, G_n\}$ be a $5$--separated family of infinite, reducible subgroups $G_i\le \Mod(S)$. 
Then for any $L >0$, there exist finite-index subgroups $G_i'\le G_i$ so that for any further subgroups $H_i \le G_i'$ which are still infinite, the family $\Hc=\{H_1, \dots. H_n\}$ is $L$--displacing.
\end{proposition}

\begin{proof}
For each $1\leq i\leq n$, let $\mu_i$ be a marking of $S$ containing $\partial G_i$. Set $\mu = \mu_1$.
By the marking complex version of the distance formula \cite[Theorem 6.12]{MasurMinsky-CurveComplexII}, there is a bound $K\ge 1$ so that $d_W(\nu,\nu') \le K$ for every subsurface $W$ and all $\nu,\nu'\in \{\mu=\mu_1,\dots, \mu_n\}$.

For the given constant $L > 0$, apply \Cref{lem:displacing_for_one_group} to each group $G_i$. Doing so, one obtains a finite-index subgroup $G_i'\le G_i$ with the property that for every $g\in G'_i$ there is some subsurface $Y$ disjoint from $\partial G_i$ with $d_Y(g\mu, \mu) \ge L + 5K$. 

We claim that for any infinite subgroups $H_i \le G'_i$, the family $\mathcal{H} = \{H_1,\dots,H_n\}$ is $L$--displacing, as desired. Indeed, we use the multicurves $\beta_i = \partial G_i$. Then the $5$--separation assumption on $\mathcal{G}$ gives $d_S(\beta_i,\beta_j)\ge 5$ for each $i\ne j$, and the containment $H_i\le G'_i\le G_i$ ensures $H_i$ stabilizes $\beta_i$. For the final condition of \Cref{def:displacing}, fix any indices $i\ne j\ne k$ and nontrivial element $h\in H_j$. Since $h\in H_j \le G'_j$, there is a subsurface $Y$ disjoint from $\partial G_j=\beta_j$ so that $d_Y(h \mu,\mu)\ge L+5K$ by the preceding paragraph. Since 
\[d_Y(\mu,\mu_i) \le K\qquad\text{and}\qquad d_{Y}(h\mu,h\mu_k) = d_{h\inv Y}(\mu,\mu_k)\le K\]
by our choice of $K$, we have $d_Y(\mu_i, h\mu_k) \ge L+3K$ by the triangle inequality. Since $\beta_l = \partial G_l$ is a subset of $\mu_l$ for $1\leq l\leq n$, we also have $d_Y(\beta_l,\mu_l)\le d_Y(\mu_l,\mu_l)\le K$ for $l\in \{i,k\}$.
Thus we conclude the required condition for displacement:
\[d_Y(\beta_i, h \beta_k) \ge d_Y(\mu_i, h\mu_k) -2K \ge L+K > L.\qedhere\]
\end{proof}

\Cref{main-subgroups} from the introduction is now an immediate consequence of \Cref{prop:finite-index-displacing-family} (applied with constant $L = \bgit + 4\bsi$) and \Cref{thm:displacing}.

\section{Right-angled Artin subgroups and the proof of \Cref{main-CLM}}
\label{sec:RAAGS}
We now turn to the task of constructing RGF right-angled Artin subgroups of the mapping class group.

	\subsection{Right-angled Artin groups and normal form} \label{ssub:RAAG_FreeProd}

Recall that a right-angled Artin group is specified by a \emph{presentation graph} $\Gamma$ with vertices $\{x_i\}_{i=1}^n$ and a symmetric set of edges $E\subset \{(x_i,x_j)\mid i\neq j\}$, so that $A(\Gamma)=\< x_1, \dots, x_n \mid [x_i,x_j] \text{ if } (x_i,x_j)\in E\>.$
Each \textit{subgraph} $\Gamma'$ of $\Gamma$ induces a subgroup $A(\Gamma')\leq A(\Gamma)$. Neither $\Gamma$ nor its subgraphs are assumed to be connected. In fact, since we require RGF subgroups to be relatively hyperbolic, we are solely interested in presentation graphs which decompose into at least two components.

Suppose $\Gamma$ decomposes into the disjoint union \[\Gamma = \Gamma_1\, \sqcup \dots \sqcup\, \Gamma_m \]  of subgraphs for $m\geq 2$. Then $A(\Gamma)$ splits as the free product $A(\Gamma_1) \ast \dots \ast A(\Gamma_m)$, and is relatively hyperbolic with respect to its factors. Conversely, except for the integers, any right-angled Artin group with connected presentation graph is \textit{not} relatively hyperbolic \cite[Proposition 1.3]{BehrstockDrutuMosher-RelHyp}.

    Let $G$ be a finitely generated group with a prescribed ordered generating set $\{x_1,\ldots,x_k\}$.
	We can 
	decompose 
each 	element $g$ of $G$
	as the concatenation of \textit{syllables} of the form $x_i^{e_i}$. That is, we write 
		\[g= s^{e_1}_{1} \dotsb s^{e_n}_{n}\] 
	where each $s_j$ is a generator $x_i$ for some $i$ and $e_j\in\mathbb Z$. 
	Such a spelling is {\em reduced} if it satisfies the following conditions: 
	\begin{itemize}
		\item 
		Each $e_i \neq 0$; otherwise, remove the entire \textit{syllable} $s^{0}_{j}$. 
		\item Each $s_j\neq s_{j+1}$; otherwise, replace $s^{e_j}_{j} s^{e_{j+1}}_{j+1}$ with $s^{e_j + e_{j+1}}_{j}$.
		\item If 
		$s_j=x_i$ and $s_{j+1}=x_l$ commute, then $i<l$; otherwise, re-order the subword $s_{j}^{e_j} s_{j+1}^{e_{j+1}}$ to $ s_{j+1}^{e_{j+1}} s_{j}^{e_j}$.
	\end{itemize}
	
	It is a fact that when $G=A(\Gamma)$ is a RAAG and the generators correspond to the vertices of $\Gamma$, every element $g$  of $G$ has a unique reduced spelling which we call the {\em normal form} of $g$ \cite{hermiller_meier} (see also \cite{charney, CLM-RAAGs_in_MCG}).

\subsection{Reducibly geometrically finite RAAGs}
As a quick application of \Cref{main-subgroups}, we explain how the technology of displacing families leads to a simple proof the following slightly weaker version of \Cref{main-CLM}:

\begin{theorem}\label{thm:weak-CLM}
	Let $\left\{f_{1},\ldots, f_{n} \right\}$ be mapping classes fully supported on subsurfaces $S_{1},\ldots , S_{n}$, respectively, with realization graph $\Gamma$. Suppose: 
	
	\begin{itemize}
		\item $\Gamma$ decomposes as the disjoint union  $\Gamma_1
		\sqcup \cdots \sqcup \Gamma_m$ of subgraphs,
		with $m \geq 2$; 
		\item the subgroup $G_k$
		of $\Mod(S)$ generated by the elements $f_i$ supported on the vertices of $\Gamma_k$
		is reducible for all $k=1,\ldots,m$; and 
		\item 
		$d_S(\partial G_\ell, \partial G_j)\geq 5$.	\end{itemize}
		Define a map $\Psi\colon A(\Gamma)\to \Mod(S)$  by 	$\Psi(x_i)=f_{i}^{N k_i}$ for some $N, k_i\in\mathbb Z_{\neq 0}$.
	Then there exists an $N>0$ such that 
 the subgroup	$\<f_1^{Nk_1}, \ldots, f_n^{Nk_n}\>$ is isomorphic to 
	$\Psi(A(\Gamma_{1})) \ast \cdots \ast \Psi(A(\Gamma_{m})) $ and 
	RGF relative to  $\{\Psi(A(\Gamma_{1})),\ldots, \Psi(A(\Gamma_{m}))\}$.
\end{theorem}
\begin{proof}
The family $\mathcal{G} = \{G_1,\dots, G_m\}$ of reducible subgroups is by assumption $5$--separated. Thus we may apply \Cref{main-subgroups} and let $G'_j \le G_j$ be the resulting finite-index subgroups. For each generator $f_i \in G_j$, there is a power $N_i$ such that $f_i^{N_i} \in G'_j$. Let $N = N_1\cdots N_n$ be the product. Then $f_i^{N k}\in G'_j$ for all $k\in \Z$; hence the subgroup $H_j=\Psi(A(\Gamma_j))$
is an infinite subgroup of $G'_j$. It now follows from \Cref{main-subgroups} that  $\langle H_1, \dots H_m \rangle \cong H_1 \ast \dots \ast H_m$ is reducibly geometrically finite. 
\end{proof}

The above theorem is nearly identical to \Cref{main-CLM} but weaker in two key ways. Firstly \Cref{main-CLM} loosens the 5--separation hypothesis $d_S(\partial G_\ell, \partial G_j)\ge 5$ on the subgroups to a mere 3--separation property $d_S(S_i, S_j)\ge 3$ on the supports of generators from distinct subgraphs (note that this is strictly weaker, since if $S_i$ is the support of a vertex generator of $\Gamma_j$, then $S_i$ and $\partial G_j$ are necessarily disjoint). Secondly, \Cref{main-CLM} strengthens the conclusion by applying to \emph{any} large powers $f_i^{p_i}$ of the generators, rather than simply those $f_i^{N k_i}$ that are multiples of some large integer $N$.

\begin{proof}[Proof of \Cref{main-CLM}]
We recall the setup of the theorem statement. Let $f_1,\dots, f_n\in \Mod(S)$ be mapping classes fully supported on an admissible collection of subsurfaces $S_1, \dots, S_n$. Let $\Gamma \colonequals \Gamma (S_1, \dots, S_n)$ be the associated realization graph and suppose it decomposes as a disjoint union $\Gamma = \Gamma_1 \sqcup \dots \sqcup \Gamma_m$ of subgraphs $\Gamma_j$ whose vertices generate reducible subgroups $G_j\le \Mod(S)$ and such that $d_S(S_i, S_k) \ge 3$ whenever vertices $f_i,f_k$ lie in distinct subgraphs.

Since $f_i$ is fully supported on $S_i$, it is pure and $S_i$ is its only domain. Hence \Cref{thm:minimal_translation_mangahas_cor} implies $d_{S_i}(f_i^n \gamma, \gamma) \ge c \abs{n}$ for any curve $\gamma$ with nontrivial projection to $S_i$. Let $D = \max_{i,j,k}\{d_{S_j}(S_i, S_k)\}$. We will show that $N =  (\bsi + 6 \bgit + D)/c$ satisfies the conclusion of the theorem. To see this, fix any tuple $(p_1,\dots, p_n)$ with $\abs{p_i}\ge N$ for each $i$.
Let $A(\Gamma)$ be the abstract RAAG on the graph $\Gamma$, with generators $x_1,\dots, x_n$ corresponding to the vertices of $\Gamma$. Let $\Psi\colon A(\Gamma)\to \Mod(S)$ the homomorphism sending $x_i$ to $f_i^{p_i}$ and set $H_j = \Psi(A(\Gamma_j))$ for $j=1,\dots, m$. We wish to verify that the group $\Psi(A(\Gamma)) = \langle f_1^{p_1},\dots, f_n^{p_n}\rangle$ is isomorphic to $H_1\ast\dots \ast H_m$ 
 and RGF with respect to the family $\{H_1,\dots, H_m\}$.

We again use the setup of \Cref{s:cayley_trees} and deduce the result from \Cref{lem:qi-to-CS-implies_RGF}: Let $T$ be the Bass--Serre tree for the abstract free product $H = H_1\ast\dots \ast H_m$ and $\phi\colon T\to \cc(S)$ the map constructed \Cref{sec:free_products_in_MCG}. 
Now, fix type--1 vertices $v,v'\in T$. If $d_T(v,v') = 2r$, then the $T$--geodesic joining them has the form $v = a_0, b_0, a_1, \dots  ,b_{r-1}, a_r = v'$ where each $b_s = \mathbf{v}(w_s)$ is the type--2 vertex labeled by some element $w_s\in H$. Since $a_s$ is adjacent to $b_{s}$ and $b_{s-1}$ when $s > 0$, we additionally have $a_s = \mathbf{v}(w_{s-1} H_{I(s)}) = \mathbf{v}(w_s H_{I(s)})$ for some index $I(s)\in \{1,\dots, m\}$. 
In particular, we have $h_s = w_{s-1}\inv w_s\in H_{I(s)}$ for $0 < s < r$. Therefore $w_0\inv w_{r-1} = h_1h_2\cdots h_{r-1}$ is evidently the unique geodesic spelling of  the element $w_0\inv w_{r-1}\in H_1\ast\dots \ast H_m$ as a product of elements of the factors $H_j$.  Thus $w_{r-1} = w_0 h_1\cdots h_{r-1}$.
	
Since $h_s \in H_{I(s)} = \Psi(A(\Gamma_{I(s)}))$, we may write $h_s = \Psi(h'_s)$ for some element $h'_s\in A(\Gamma_{I(s)})$. Let us write $h'_s = y_{s,1}   \dotsb y_{s,k_s}$ in normal form for $A(\Gamma_{I(s)})$, where each $y_{s,i}$ equals $x^e$ for some vertex generator $x\in \Gamma_{I(s)}$ and nonzero power $e\in \Z$, and where the generators associated to adjacent letters $y_{s,i}$ and $y_{s,i+1}$ are distinct with the lower-indexed generator coming first in the case they commute. We concatenate these words and re-index 
\[ 
	g_1' \dotsb g_\ell' = \left(y_{1,1}\dotsb y_{1,k_1}\right) \left(y_{2,1}\dotsb y_{2,k_2}\right) \dotsb \left(y_{r-1,1}\dotsb y_{r-1,k_{r-1}}\right) 
\]
to get a normal-form word in $A(\Gamma)$ of length $\ell = k_1 + \dots + k_{r-1}$. Taking the image now expresses $h_1\dotsb h_{r-1} = \Psi(g'_1\dots g'_\ell) = g_1\dotsb g_\ell$ as a product of elements $g_t = \Psi(g'_t) = \Psi(x_{\nu(t)}^{e_t})$, each of which has the form  $g_t = f_{\nu(t)}^{q_t}$ for some index $\nu(t)\in \{1,\dots, n\}$ and exponent satisfying $\abs{q_t} = \abs{e_t p_{\nu(t)}} \ge N$.
For $0 < s < r$, let $\sigma(s) = k_1 + \dots + k_{s-1} + 1$ be the index of the first letter of the subword $h_s = \Psi(h_s')$ in the above spelling $h_1\dotsb h_{r-1} = g_1\dotsb g_\ell$; so that $g_{\sigma(s)} = \Psi(y_{s,1})$.

We now define a list $Y_1,\dots, Y_\ell$ of subsurfaces by declaring $Y_t = w_0 g_1\dotsb g_t \cdot S_{\nu(t)}$, where $S_{\nu(t)}$ is the support of the pure element $g_t = f_{\nu(t)}^{q_t}$. 
Since $g_t$ preserves $S_{\nu(t)}$, we additionally note that $Y_t = w_0 g_1\cdots g_{t-1} \cdot S_{\nu(t)}$ and that this translated surface $Y_t$ is the support of the conjugate $\beta_t = (w_0g_1\cdots g_{t-1})g_t(w_0g_1\cdots g_{t-1})\inv$. Note also that these conjugates satisfy 
\[
	\beta_t \dotsb \beta_1 w_0 = w_0 g_1 \dotsb g_t\quad\text{for all}\quad 0<t\le \ell.
\]

To round out notation, let us also set $Y_0 = w_0 \cdot S_{\nu(0)}$ and $Y_{\ell+1} = w_0g_1\dotsb g_\ell \cdot S_{\nu(\ell+1)}$, where $\nu(0), \nu(\ell+1)\in \{1,\dots, n\}$ are any indices so that the subsurfaces $S_{\nu(0)}$, $S_{\nu(\ell+1)}$ lie in the respective subgraphs $\Gamma_{I(0)}$ and $\Gamma_{I(r)}$ of $\Gamma$. In particular, since $f_{\nu(0)}\in H_{I(0)}$ preserves $\partial H_{I(0)}$ and is fully supported on $S_{\nu(0)}$, the multicurves $\partial S_{\nu(0)}$ and $\partial H_{I(0)}$ are disjoint. Similarly $\partial S_{\nu(\ell+1)}$ and $\partial H_{I(r)}$ are disjoint. 
Recalling that $v = a_0 = \mathbf{v}(w_0 H_{I(0)})$ and $v' = a_r = \mathbf{v}(w_{r-1}H_{I(r)})$ where $w_{r-1} = w_0 h_1\dotsb h_{r-1} = w_0 g_1\dotsb g_\ell$, it follows that $\partial Y_0$ is disjoint from $w_0 \cdot \partial H_{I(0)} = \phi(v)$ and that $\partial Y_{\nu(\ell+1)}$ is disjoint from $w_{r-1}\cdot \partial H_{I(r))} = \phi(v')$. We also note that
\begin{align*}
d_S(Y_0, Y_1) &= d_S(w_0 \cdot S_{\nu(0)}, w_0 \cdot S_{\nu(1)}) = d_S(S_{\nu(0)}, S_{\nu(1)})\le L\qquad\text{and}\\
d_S(Y_{\ell}, Y_{\ell+1}) &= d_S(w_0 g_1\dotsb g_\ell \cdot S_{\nu(\ell)} , w_0 g_1\dotsb g_\ell \cdot S_{\nu(\ell+1)}) =d_S(S_{\nu(\ell)}, S_{\nu(\ell+1)})\le L,
\end{align*}
where $L = \max_{i,j}d_S(S_i, S_j)$. 
Therefore, by the triangle inequality, we have
\begin{equation}
\label{eqn:related_to_Y_t_distance}
d_S(\phi(v), \phi(v')) \ge d_S(Y_0, Y_{\ell+1}) -2 \ge d_S(Y_1, Y_\ell) -2 -2L.
\end{equation}
It hence suffices to give a lower bound on $d_S(Y_1,Y_\ell)$. We will accomplish this by applying \Cref{cor:general_projections_persist} to the sequence $Y_1,\dots, Y_\ell$ and \Cref{cor:projections_persist} to the subsequence $Y_{\sigma(1)},\dots, Y_{\sigma(r-1)}$.  To enable this, we first establish some facts:

\begin{claim}
\label{claim:technical_bit_of_ThmA}
Fix $1\le j\le \ell$. 
If $\iota(j) < t < \tau(j)$ then:
\begin{itemize}
\item $g_j$ and $g_t$ commute and lie in a common subword $h_s$ of $h_1\dotsb h_{r-1}$; that is $\sigma(s) \le j,t < \sigma(s+1)$ for some $0 < s < r$.
\item $\beta_t$ fixes $Y_j$ and, if $t\ne j$, then $Y_t$ is disjoint from $Y_j$ and $\beta_t$ does not change projections to $Y_j$, in that $\pi_{Y_j}(\beta_t \gamma) = \pi_{Y_j}(\gamma)$ for every multicurve $\gamma$ on $S$.
\end{itemize}
\end{claim}
\begin{proof}
\renewcommand{\qedsymbol}{$\Diamond$}
We only consider the case $j \le t < \tau(j)$, with the alternative $\iota(j) < t \le j$ being symmetric.
We proceed by inducting on $t$, with the base case $t = j$ being trivial. So, suppose $j < t < \tau(j)$ and that the claim holds for all $j\le t'<t$. 
Recall that $Y_{t} = \beta_{t-1}\cdots\beta_1w_0 \cdot S_{\nu(t)}$. The induction hypothesis also implies $Y_j = \beta_{t-1}\cdots \beta_j \cdot Y_j = \beta_{t-1}\cdots \beta_1 w_0 \cdot S_{\nu(j)}$. Therefore the fact that $Y_j$ and $Y_{t}$ do not overlap implies $S_{\nu(j)}$ and $S_{\nu(t)}$ do not overlap.

By induction we know the letters $g_j,\dots,g_{t-1}$ lie in a common subword $h_s$. It follows that $g_t$ must also lie in this subword, since all letters in the next subword $h_{s+1}\in H_{I(s+1)}$ are images of generators from distinct subgraphs $\Gamma_{I(s+1)} \ne \Gamma_{I(s)}$ and hence are supported on surfaces that overlap $S_{\nu(j)}$. Additionally, our admissibility hypothesis ensures distinct surfaces in the family $\{S_1,\dots, S_n\}$ are not nested. Therefore, since they do not overlap, $S_{\nu(j)}$ and $S_{\nu(t)}$ must either be disjoint or equal. But if they are equal, then $g_j$ and $g_t$ are both powers of the same generator $f_{\nu(j)}=f_{\nu(t)}$ and hence $g_t$ also commutes with the letters $g_{j},\dots, g_{t-1}$. This violates our normal form assumption on the word $h_s'\in A(\Gamma_{I(s)})$ wherein commuting letters must appear in order of increasing index. Therefore $S_{\nu(j)}$ and $S_{\nu(t)}$ are disjoint. It now follows that $g_j$ and $g_t$ commute and that $Y_t$ is disjoint from $Y_j$. Moreover, since $\beta_t$ is fully supported on $Y_t$ and $Y_j$ is not the annulus about a boundary component of $Y_t$, it follows that $\beta_t$ preserves $Y_j$ and does not change projections to $Y_j$.
\end{proof}

\begin{claim}
\label{claim:satisfy_hypothesis_of_corollary}
The sequence $Y_1,\dots Y_\ell$ satisfies the hypothesis of \Cref{cor:general_projections_persist}.
\end{claim}
\begin{proof}
\renewcommand{\qedsymbol}{$\Diamond$}
Fix some index $1 < j< \ell$ so that $\iota(j)$ and $\tau(j)$ both exist. To ease notation, set $\beta = \beta_{j-1}\dotsb \beta_1 w_0$. Then we may write
\begin{align*}
Y_{j} = \beta \cdot S_{\nu(j)},\quad 
Y_{\iota(j)} &= (\beta_{\iota(j)+1}\inv\dotsb \beta_{j-1}\inv)\beta \cdot S_{\nu(\iota(j))},\\
\text{and}\quad
 Y_{\tau(j)} &= (\beta_{\tau(j)-1}\dotsb \beta_{j+1})\beta_j \beta \cdot S_{\nu(\tau(j))}.
\end{align*}
By \Cref{claim:technical_bit_of_ThmA}, none of the elements $\beta_{\iota(j)+1},\dots, \beta_{j-1},\beta_{j+1},\dots,\beta_{\tau(j)-1}$ change projections to $Y_j$. Therefore the prefixes $(\beta_{\iota(j)+1}\inv\dotsb \beta_{j-1}\inv)$ and $(\beta_{\tau(j)-1}\dotsb \beta_{j+1})$ above do not effect projections and we find that
\begin{align*}
d_{Y_j}(Y_{\iota(j)}, Y_{\tau(j)})
&= d_{\beta \cdot S_{\nu(j)}}(\beta \cdot S_{\nu(\iota(j))}, \beta_j \beta \cdot S_{\nu(\tau(j))})
= d_{S_{\nu(j)}}(S_{\nu(\iota(j))}, g_j \cdot S_{\nu(\tau(j))})
\end{align*}
where here have used the observation $\beta\inv \beta_j \beta = g_j$. By hypothesis, we have $g_j = f_{\nu(j)}^{q_t}$ where $\abs{q_t} \ge N = (\bsi + 6\bgit+D)/c$ and $d_{S_{\nu(j)}}(S_{\nu(\iota(j))}, S_{\nu(\tau(j))}) \le D$. Therefore, since $f_{\nu(j)}$ is fully supported on $S_{\nu(j)}$, by \Cref{thm:minimal_translation_mangahas_cor} we conclude
\[d_{Y_j}(Y_{\iota(j)},Y_{\tau(j)}) \ge d_{S_{\nu(j)}}(S_{\nu(\tau(j))}, f_{\nu(j)}^{e_t} S_{\nu(\tau(j))}) - D \ge Nc - D = \bsi + 6\bgit.\qedhere \]
\end{proof}

\begin{claim}
\label{claim:subsequence_big_distance}
For each $s$ we have $d_S(Y_{\sigma(s-1)}, Y_{\sigma(s)})\ge 3$ and thus $Y_{\sigma(s-1)}\pitchfork Y_{\sigma(s)}$.
\end{claim}
\begin{proof}
\renewcommand{\qedsymbol}{$\Diamond$}
Starting with $a_1 = \sigma(s-1)$, build a sequence $a_1,a_2,\dots$ by recursively setting $a_{i+1} = \tau(a_i)$. This must eventually yield $a_{k+1}= \tau(a_k) = \sigma(s)$ for some $k$, since otherwise we have $a_k < \sigma(s) < \tau(a_k)$ which, by \Cref{claim:technical_bit_of_ThmA} would imply $g_{a_k}$ and $g_{\sigma(s)}$ lie in the same subword of $h_1\dotsb h_{r-1}$, violating the fact that  $g_{\sigma(s)}$ is the first letter in the subword $h_s$. Thus \Cref{claim:technical_bit_of_ThmA} now tells us the elements $\beta_{a_{k+1}},\dots, \beta_{\sigma(s)-1}$ preserve the surface $Y_{a_k}$ and that $S_{\nu(a_k)}$ and $S_{\nu(\sigma(s))}$ lie in distinct subgraphs $\Gamma_{I(s-1)}\ne \Gamma_{I(s)}$. 
Therefore $Y_{a_k} = \beta_{\sigma(s)-1}\dotsb \beta_{a_{k+1}}\cdot Y_{a_k}$ and we have
\begin{align*}
d_S(Y_{a_k},Y_{\sigma(s)}) &= d_S(\beta_{\sigma(s)-1}\dotsb \beta_1 w_0 \cdot S_{\nu(a_k)}, \beta_{\sigma(s)-1}\dotsb \beta_1 w_0 \cdot S_{\sigma(s)})\\
&= d_S(S_{\nu(a_k)}, S_{\nu(\sigma(s))}) \ge 3
\end{align*}
by our assumption on the family $\{S_1,\dots, S_n\}$ and subgraphs $\Gamma_1,\dots,\Gamma_m$.

Now, by construction the subsequence $\sigma(s-1) = a_1 , \dots, a_{k+1} = \sigma(s)$ satisfies $Y_{a_i} \pitchfork Y_{a_{i+1}}$ for each $i$. Since $Y_1, \dots, Y_\ell$ satisfies the hypothesis of \Cref{cor:general_projections_persist} by \Cref{claim:satisfy_hypothesis_of_corollary}, we may apply  \Cref{cor:projections_persist} to the subsequence and conclude that
\[d_S(Y_{\sigma(s-1)}, Y_{\sigma(s)}) = d_S(Y_{a_1}, Y_{a_{k+1}}) \ge d_S(Y_{a_k}, Y_{a_{k+1}}) \ge 3.\qedhere\]
\end{proof}

With these claims in hand, it is now trivial to complete the proof of the theorem. By \Cref{claim:satisfy_hypothesis_of_corollary} and \Cref{claim:subsequence_big_distance} we may apply \Cref{cor:general_projections_persist} to the sequence $Y_1,\dots, Y_\ell$ to conclude the subsequence $Y_{\sigma(1)},Y_{\sigma(2)},\dots,Y_{\sigma(r-1)}$ satisfies the hypothesis of \Cref{cor:projections_persist}. Therefore, since $d_S(Y_{\sigma(s-1)}, Y_{\sigma(s)})\ge 3$ for each $s$, we have $d_S(Y_{\sigma(1)}, Y_{\sigma(r-2)}) \ge r-2$. Notice that $Y_1 = Y_{\sigma(1)}$. If $Y_{\sigma(r-1)}\pitchfork Y_\ell$, we may also apply \Cref{cor:projections_persist} to $Y_{\sigma(1)},\dots, Y_{\sigma(r-1)}, Y_\ell$ to conclude $d_S(Y_1, Y_\ell) \ge d_S(Y_{\sigma(1)},Y_{\sigma(r-1)})\ge r-2$. Otherwise $d_S(Y_{\sigma(r-1)}, Y_\ell)\le 1$ and we have $d_S(Y_1,Y_\ell)\ge r-3$. In either case,  together with (\ref{eqn:related_to_Y_t_distance}), this gives the following bound needed to apply \Cref{lem:qi-to-CS-implies_RGF}:
\[d_S(\phi(v),\phi(v')) \ge r-3 - 2 - 2L = \frac{1}{2}d_T(v,v') - 5-2L.\qedhere\]
\end{proof}

\section{Separability, misalignment and the proof of \Cref{main-Loa}}\label{s:separability} 

Now, we'll start with a family $\Hc = \{H_1,\dots,H_n\}$ of infinite, torsion-free, reducible subgroups $H_i\le \Mod(S)$ in aims of proving \Cref{main-Loa}. We'll address the necessity of the added torsion-free assumption in \Cref{E:example_torsion_free}, but the theorem has two special hypotheses on the family. The first is that the collection is $D$--separated (\Cref{def:separated}) meaning $d_S(\del H_i, \del H_j) \ge D$ for all distinct $i,j$. 
The second condition is the following:

\begin{definition}(Misalignment)
\label{def:misaligned}
A collection $\Hc= \{H_1,\dots,H_n\}$ of infinite reducible subgroups $H_i\le \Mod(S)$ is \textbf{$A$--misaligned} if their $\Cc(S)$ Gromov products satisfy
\[		
	(\del H_i \mid \del H_j)_{\del H_{k}}\geq A \quad \text{ for all distinct indices }i,j,k.
\]
\end{definition}

\begin{remark}
\label{rem:gromov_product_multicurves}
Note that the Gromov product here is by definition
\[(\del H_i \mid \del H_j)_{\del H_{k}} := \frac{1}{2}\left(d_S(\partial H_k, \partial H_i) + d_S(\partial H_j, \partial H_k) - d_S(\partial H_i, \partial H_j)\right),\]
where each distance is the diameter of the union of the two sets. Since each multicurve $\partial H_i$ is a diameter at most $1$ subset of $\Cc(S)$, this quantity lies within 2 of $(\alpha_i \mid \alpha_j)_{\alpha _k}$ for any particular choice of elements $\alpha_i,\alpha_j,\alpha_k$ of these multicurves.
Intuitively, the misalignment condition says that the reducing multicurves $\partial H_i,\partial H_j,\partial H_k$ form a tripod in the curve graph. See \Cref{fig:misalignment}. 
\end{remark}

\begin{figure}[h]

%
%

\begin{tikzpicture}[scale=.7,rotate=20]

\begin{scope}[dashed, thick, red!55!blue]
\draw (170:4) coordinate(1) -- 
node[pos=.7, yshift=-1.5em] {$\gtrsim A$} 
(0,0); 
\draw (80:3) coordinate(2) -- 
node[pos=.7, xshift=1.75em] {$\gtrsim A$} 
(0,0); 
\draw (-60:4.5) coordinate(3) -- 
node[pos=.7, xshift=1.85em] {$\gtrsim A$} 
(0,0); 
\end{scope}

\draw (1) .. controls (0,0) .. (2);
\draw (1) .. controls (0,0) .. (3);
\draw (2) .. controls (0,0) .. (3);

  \def\s{.05}
  \draw[fill] (1) circle (\s) node[below] {\small $\partial H_1$};
  \draw[fill] (2) circle (\s) node[above] {\small $\partial H_2$};
  \draw[fill] (3) circle (\s) node[below] {\small $\partial H_3$};

  \def\ss{.75}
  \begin{scope}[thick]
  \draw[red] (1) circle (\ss);
  \draw[blue] (2) circle (\ss);
  \draw[orange] (3) circle (\ss);
  \end{scope}

\end{tikzpicture}

\caption{The configuration of reducing multicurves in $\Cc$ for a misaligned triple $H_1,H_2,H_3$. Each $\partial H_i$ is roughly distance at least $A$ from a geodesic $\partial H_j$ to $\partial H_k$, hence the triple of multicurves forms a tripod. }
\label{fig:misalignment}
\end{figure}

\begin{remark}
We note that $A$--misaligned implies $(2A-4-16\delta)$--separated. To see this, for any pairwise distinct triple $i,j,k$ pick any choice of curves $\alpha_\ell\in\partial H_\ell$.  
Choose a geodesic triangle with vertices $\alpha_i,\alpha_j,\alpha_k$, and let $p_i,p_j,p_k$ be vertices of an inner triangle where $p_\ell$ is opposite $\alpha_\ell$ (in other words, $p_k$ is on the opposite edge $[\alpha_i,\alpha_j]$, etc). Then $\diam(\{\alpha_i,\alpha_j,\alpha_k\})\leq4\delta$ by hyperbolicity of $\mathcal C(S)$. The $A$--misaligned condition and \eqref{eqn:gromov-product-to-geod} give that $d(\alpha_{\ell_1},[\alpha_{\ell_2},\alpha_{\ell_3}])\geq A-2-4\delta$ for all $\{\ell_1,\ell_2,\ell_3\}=\{i,j,k\}$. 
Then 
\begin{align*}
d(\partial H_i,\partial H_j) \geq d(\alpha_i,\alpha_j)&=d(\alpha_i,p_k)+d(p_k,\alpha_j) \\
& \geq d(\alpha_i,p_i)+d(p_j,\alpha_j)-8\delta \\
& \geq 2A-4-16\delta.
\end{align*}
\end{remark}

Our proof of \Cref{main-Loa} will roughly follow the argument of Loa \cite[Theorem 1.1]{Loa}, with adaptations to handle the more general aspects of our setup, namely arbitrary reducible subgroups and a larger number of them. 
The first step is to show that each element of a reducible group $H$ moves curves in $\Cc(S)$ a definite fraction of their distance to $\partial H$. This mirrors Lemma 3.1 of Loa's argument, but because our reducible group $H$ need not be a multitwist group, we use \Cref{thm:minimal_translation_mangahas_cor} instead of a computation in the curve graph of an annulus.

\begin{lemma}
\label{lem:definite_distance}
There is a constant $K\ge 1$, depending only on $S$, with the following property:
Let $H\le \Mod(S)$ be a reducible subgroup and $ g\in H$ an infinite order element. Then every multicurve
 $\alpha$ on $S$ satisfies
\[d_S(\alpha,  g\alpha) \ge \frac{d_S(\alpha,\partial H) - 3}{K} .\]
\end{lemma}
\begin{proof}
As noted in \Cref{ssub:MCG_Reducibles}, there is a uniform power $N\ge 1$, depending only on $S$, so that $ g^N$ is pure, that is, in normal form. Since $ g$ has infinite order, $ g^N$ is also nontrivial. Let $Y$ be any domain of $ g^N$. 
By \Cref{thm:minimal_translation_mangahas_cor} there is a uniform constant $c>0$, depending only on $S$, such that
\begin{equation}
\label{eqn:translation_length_bound}
d_Y(\beta,( g^N)^m\beta)\geq mc
\end{equation}
for every $m\ge 1$ and every curve $\beta$ that projects to $Y$. In particular, $\partial Y$ must  lie in the reducing system $\partial H'$ for the infinite cyclic reducible group $H' = \langle  g^N\rangle \le H$, since $\partial Y$ is invariant under $ g^N$ and all curves cutting $\partial Y$ evidently have infinite $H'$--orbit. Since the components of $\partial H$ clearly have finite $H'$--orbit and so lie in $R(H')$, we conclude that $d_{S}(\partial Y, \partial H)\le 1$ by definition of $\partial H'$.

Now consider any multicurve $\alpha$. If $d_S(\alpha,\partial H)\le 3$ the claim is trivially true. So we suppose $d_S(\alpha,\partial Y) \ge 4$ which implies that each component of $\alpha$ intersects $\partial Y$ and so projects to $Y$. Choose $m$ so that $mc-2 > \bgit$, the constant from the Bounded Geodesic Image Theorem (\Cref{thm:BddGeodesicImage}).

Next let $\beta\in \alpha$ and $\beta'\in  g^{mN}\alpha$ realize the diameter $d_S(\alpha, g^{mN}\alpha)$. Note that $d_Y(\beta', g^{mN}\beta)\le 2$ by (\ref{eqn:multicure_bdd_projection}) since these curves are disjoint. Hence by \eqref{eqn:translation_length_bound} we get
\[ d_Y(\beta,\beta') \ge  d_Y(\beta, g^{mN}\beta) - 2 \ge mc-2 > \bgit.\]

The contrapositive of \Cref{thm:BddGeodesicImage} now implies that any geodesic $[\beta,\beta']\subset \Cc(S)$ contains a point $\gamma$ with empty projection to $Y$. 
In particular, $\gamma$ is disjoint from $\partial Y$ which is in turn disjoint from $\partial H$, so we get $d_S(\gamma,\partial H)\le 2$.
Consequently, 
\[d_S(\alpha,\partial H) \le d_S(\beta,\partial H) +1 \le d_S(\beta,\gamma) +3.\]
Similarly, since $ g^{mN}$ fixes $\partial H$ by definition,  we find that
\[ d_S(\alpha,\partial H) = d_S( g^{mN}\alpha,\partial H) \le d_S(\beta',\partial H)+1 \le d_S(\beta',\gamma)+3.\]
Combining these, and using the fact that $\gamma\in [\beta,\beta']$ we now conclude that
\[d_S(\alpha, g^{mN}\alpha) 
= d_S(\beta,\beta') 
= d_S(\beta,\gamma) + d_S(\gamma,\beta')
\ge 2(d_S(\alpha,\partial H) - 3).\]

On the other hand, $\{\alpha,  g\alpha, \dots,  g^{mN} \alpha\}$ gives path from $\alpha$ to $ g^{mN}\alpha$ with $mN$ segments of equal length $d_S(\alpha, g \alpha)$. Hence by the triangle inequality we conclude
\[2(d_S(\alpha,\partial H) - 3) \le d_S(\alpha, g^{mN}\alpha) \le mN d_S(\alpha, g\alpha),\]
which proves the lemma with constant $K = \frac{mN}{2}$, depending only on $S$.
\end{proof}

The next step is to use hyperbolicity of $\Cc(S)$ to achieve an upper bound for the Gromov product of $\alpha$ and $ g\alpha$ with respect to $\partial H$ (compare \cite[Lemma 3.2]{Loa}).

\begin{lemma}
\label{lem:loa-bound-gromov-product}
There is a constant $K'>0$ satisfying the following: Let $ g$ be an infinite order element of a reducible subgroup $H\le \Mod(S)$. Then every multicurve $\alpha$ satisfies $(\alpha\mid  g\alpha)_{\partial H}\le K'$.
\end{lemma}
\begin{proof}
Let $\delta$ be the hyperbolicity constant of $\Cc(S)$. Fix a component $\gamma$ of $\partial H$ and any component $\alpha_0$ of $\alpha$. Since $g$ preserves $\partial H$, which is a multicurve, we have $d_S(\gamma,g\gamma)\leq 1$. 
Fix a geodesic $[\gamma,\alpha_0]$ from $\gamma$ to $\alpha_0$. 
The image $ g[\gamma,\alpha_0]$ gives a geodesic from $g \gamma$ to $ g\alpha_0$ which we denote $[g\gamma, g\alpha_0]$. We also fix a geodesic $[\alpha_0, g\alpha_0]$, giving us a geodesic triangle in $\Cc(S)$ with vertices $\alpha_0,\gamma, g\alpha_0$.

By the inner triangle formulation of hyperbolicity, there exist points $z\in [\alpha_0, g\alpha_0]$, $x\in [\gamma,\alpha_0]$, and $y\in [\gamma, g\alpha_0]$, as illustrated in \Cref{fig:LoaStep2}, whose pairwise distances are at most $4\delta$. 
Let $y'$ be the closest point projection of $gx$ onto the geodesic $[\gamma,g\alpha_0]$. 
By the thin triangles definition of hyperbolicity, since $d(\gamma,g\gamma)\leq 1$, we must have $d(gx,y')\leq {4\delta}+1$.
By definition  of inner triangle (see \Cref{S:hyperbolic}), we have
$d_S(gx,g\gamma)=d_S(x,\gamma)=d_S(y,\gamma)$. 
Since $d_S(\gamma, g\gamma)\le 1$, it follows that $\abs{d_S(gx,\gamma) - d_S(y,\gamma)}\le 1$. The bound $d_S(gx,y')\le 4\delta+1$ thus implies $\abs{d_S(y',\gamma) - d_S(y,\gamma))} \le 4\delta+2$. But since $y,y'$ both lie on a geodesic $[\gamma, g\alpha_0]$ starting at $\gamma$, this implies $d_S(y,y')\le 4\delta+2$.
The triangle inequality now implies $d(x,gx)\leq {12\delta}+3$.

	\begin{figure}[h]
	
	\begin{tikzpicture}[scale=.9]
  \begin{scope}[]
    \draw (-4,0) coordinate (alpha0);
  \draw (0,-3) coordinate (gamma);
  \draw (1,-3) coordinate (ggamma);
  \draw (4,0) coordinate (galpha0);
  \draw (alpha0) to [out=-20,in=110]  coordinate[pos=.75] (x) coordinate[pos=.9] (d1) (gamma);
  \draw (alpha0) to[out=-20, in=200] coordinate[pos=.5] (z) (galpha0);
  \draw (gamma) to[out=70, in=200] coordinate[pos=.2] (y') coordinate[pos=.25] (y) coordinate[pos=.1] (d2) 
   (galpha0);
  \draw (gamma)--(ggamma);
  \draw (ggamma) to[out=80, in=200] 
  coordinate[pos=.25] (gx)
  coordinate[pos=.1] (d3)
  (galpha0);
  \end{scope}

  \begin{scope}[thin]
  \draw (x) to[out=30,in=150] (y);
  \draw (z) to[out=-80,in=150] (y);
  \draw (x) to[out=30,in=-100] (z);
  \end{scope}

  \begin{scope}
  \def\d{.1}
  \draw (d1)++(30:\d) -- (d1); 
  \draw (d1)++(180+30:\d) -- (d1); 

  \draw (d2)++(180-30:\d) -- (d2); 
  \draw (d2)++(-30:\d) -- (d2);

  \draw (d3)++(180-20:\d) -- (d3); 
  \draw (d3)++(-20:\d) -- (d3); 
  \end{scope}

  \def\s{.05}
  \draw[fill] (alpha0) circle (\s) node[left] {$\alpha_0$};
  \draw[fill] (gamma) circle (\s) node[below] {$\gamma$};
  \draw[fill] (ggamma) circle (\s) node[below] {$g\gamma$};
  \draw[fill] (gx) circle (\s) node[right] {$ g x$};
  \draw[fill] (x) circle (\s) node[below left] {$x$};
  \draw[fill] (y) circle (\s) node[above] {$y$};
  \draw[fill] (z) circle (\s) node[above] {$z$};
  \draw[fill] (galpha0) circle (\s) node[right] {$ g\alpha_0$};
  \draw[fill] (y') circle (\s) node[left] {$y'$};

  \draw (gamma)++(0,-.5) coordinate (sgamma);
  \draw (ggamma)++(0,-.5) coordinate (sggamma);

  \draw [decorate,decoration={brace,amplitude=3pt}]
  (sggamma)--(sgamma)
  node [midway,xshift=0pt,yshift=-10pt] {$\leq 1$};

\end{tikzpicture}
	\caption{The arrangement of curves in $\Cc(S)$ for \Cref{lem:loa-bound-gromov-product}.}
		\label{fig:LoaStep2}
	\end{figure}	

Applying \Cref{lem:definite_distance} to the curve $x$, we now find that
\[
\frac{d_S(x,\partial H) -3}{K} \le d_S(x, g x) \le {12\delta}+3.
\]
Therefore $d_S(x,\gamma) \le d_S(x,\partial H) \le  ({12\delta}+3)K+ 3$. Hence by the triangle inequality $d_S(\gamma,z) \le 4\delta+ ({12\delta}+3)K+3$. Finally, by (\ref{eqn:gromov-product-to-geod}) and \Cref{rem:gromov_product_multicurves} we conclude
\[(\alpha \mid  g\alpha)_{\partial H} \le (\alpha_0 \mid g\alpha_0)_{\gamma}  +2 \le d_S(\gamma,z) + 2 \le 4\delta+({12\delta}+3)K+5.\qedhere\]
\end{proof}

With these lemmas in hand, we are now prepared for the proof of \Cref{main-Loa} from the introduction. Our goal is to show there are universal constants $D,A$ such that whenever the reducible subgroups $\{H_1, \ldots, H_n\}$ are $D$--separated and $A$--misaligned, the subgroup they generate is reducibly geometrically finite and is isomorphic to the free product $H_1\ast\dots\ast H_n$.

\begin{proof}[Proof of \Cref{main-Loa}]
Let $K'$ be the constant from \Cref{lem:loa-bound-gromov-product}. Let $A := K'+5+\delta$ and $D := 4K'+11+28\delta$. Consider any $D$--separated and $A$--misaligned collection $\Hc = \{H_1,\dots,H_n\}$ of infinite, torsion-free, reducible subgroups $H_i\le\Mod(S)$. As in \Cref{s:cayley_trees}, let $T$ be the Bass--Serre tree for the abstract free product $H = H_1\ast\dots \ast H_n$ and $\phi\co T \to \mathcal{C}(S)$ the equivariant map with respect to the action of $H$ on $\cc(S)$ induced by $\Phi\colon H\to G = \langle H_1,\dots, H_n\rangle \le \Mod(S)$. To prove the theorem, it suffices to establish the lower bound required by \Cref{lem:qi-to-CS-implies_RGF}. For this, we will use the local-to-global principle (\Cref{lem:local-to-global}). 

So, consider any type--1 vertices $v,v'\in T$ and let $v= a_0,a_2,\dots, a_r = v'$ be the sequence of type--1 vertices along the geodesic $[v,v']$ in $T$. 
For each $0\le s \le r$,  choose $\beta_s$ to be any component of the multicurve $\phi(a_s)$.
As in the proof of \Cref{thm:displacing}, each consecutive triple has the form $a_{s-1} = {\bf v}(g H_i)$, $a_s = {\bf v}(gH_j) = {\bf v}(g'H_j)$, and $a_{s+1} = {\bf v}(g'H_k)$ for some $i\ne j \ne k$ and $g,g' \in H$ with $h = g\inv g'\in H_j$.
 Thus by definition $\phi(a_{s-1}) = g\cdot \partial H_i$, $\phi(a_s) = g\cdot \partial H_j$, and $\phi(a_{s+1}) = g'\cdot \partial H_k$. In particular, there are unique components $\alpha_i\in \partial H_i$, $\alpha_j\in \partial H_j$ and $\alpha_k\in \partial H_k$ so that
\[\beta_{s-1} = g\cdot \alpha_i,\quad \beta_s = g\cdot \alpha_j,\quad\text{and}\quad \beta_{s+1} = g'\cdot \alpha_k = gh\cdot \alpha_k.\]
Therefore we see that
\[(\beta_{s-1}\mid \beta_{s+1})_{\beta_s} = (g\cdot \alpha_i \mid gh\cdot \alpha_k)_{g\cdot \alpha_j} 
= (\alpha_i \mid h\cdot \alpha_k)_{\alpha_j}\]
where $h\in H_j$.
 Then by the definition of $\delta$--hyperbolicity (\ref{eqn:hyperblicity}), \Cref{rem:gromov_product_multicurves}, and \Cref{lem:loa-bound-gromov-product},  we observe:
\begin{align*}
\min\left\{(\alpha_i \mid h\cdot \alpha_k)_{\alpha_j}, (h\cdot \alpha_k \mid h\cdot \alpha_i)_{\alpha_j} \right\} - \delta -2
&\le (\alpha_i \mid h\cdot \alpha_i)_{\alpha_j} -2\\
&\le (\alpha_i \mid h\cdot \alpha_i)_{\partial H_j} \le K'.
\end{align*}
Because $h\in H_j$ acts isometrically on $\Cc(S)$ with $h\cdot\partial H_j = \partial H_j$, we also know that
\[(h\cdot \alpha_i\mid h\cdot  \alpha_k)_{\alpha_j} +{2}
\ge (h\cdot \partial H_i \mid h\cdot \partial H_k)_{h\cdot \partial H_j} = (\partial H_i \mid \partial H_k)_{\partial H_j}.\]

If $i\ne k$, this rightmost quantity is at least $A$ by  misalignment, and if $i= k$ it is 
\[(\partial H_i \mid \partial H_i)_{\partial H_j} \ge d_S(\partial H_i, \partial H_j) - 1 \ge D-1\]
by separation. In either case, since $\min\{A,D-1\}-{2} > K'+2+\delta$, the minimum above must be achieved by $(\alpha_i \mid h\cdot \alpha_k)_{\alpha_j}$, and we conclude
\begin{equation*}
(\beta_{s-1}\mid \beta_{s+1})_{\beta_s} = (\alpha_i \mid h\cdot \alpha_k)_{\alpha_j} \le K'+2+\delta.
\end{equation*}
By $D$--separation we also have
\[d_S(\beta_{s-1},\beta_s)\ge d_S(g\cdot \partial H_i, g\cdot \partial H_j)-2\ge D-2 > 4(K'+2+\delta)+24\delta.\]
Hence we may apply \Cref{lem:local-to-global} to the sequence $\beta_0,\dots,\beta_r$ to conclude
\[d_S(\phi(v),\phi(v')) \ge d_S(\beta_0,\beta_r) \ge \frac{1}{2}\sum_{s=1}^r d_S(\beta_{s-1},\beta_s)
\ge \frac{D-2}{2}r = \frac{D-2}{4} d_T(v,v').\]
The conclusion of the theorem therefore follows from \Cref{lem:qi-to-CS-implies_RGF}.
\end{proof}

\section{Examples}\label{examples} 
In this section, we'll include examples of collections of reducible subgroups 
which satisfy some or all of the separating and misaligned conditions in our results. \Cref{ex:misaligned-separated} verifies that our results are not vacuous, while Examples \ref{ex:free_product}, \ref{E:qi_embedding_fail}, \ref{E:example_torsion_free} confirm that our assumptions are necessary for the conclusions of \Cref{main-Loa} and \Cref{main-subgroups}. 
It may be helpful for the reader to consider groups generated by a collection of mapping classes fully supported on a collection of subsurfaces, as in \Cref{remark:reducibleExamples} or the setup of Clay-Leininger-Mangahas \cite{CLM-RAAGs_in_MCG}. Our \Cref{main-subgroups} additionally allows for torsion elements.

\subsection{Examples which satisfy the hypotheses of \Cref{main-Loa}}

As $\Cc(S)$ is infinite-diameter, for each $D>0$, one can construct collections of reducible subgroups which are $D$--separated with ease. For example, consider a collection of curves $\alpha_1, \ldots, \alpha_k$ with $\min{d_S(\alpha_i, \alpha_j)} > D+2$. It follows that for any infinite $H_i \leq \stab(\alpha_i)$, the collection $\{H_1, \ldots H_k\}$ is $D$--separated. We'll upgrade this setup to produce a large collection of examples which are also arbitrarily misaligned.

\begin{proposition} \label{ex:misaligned-separated}
		Let $D'>8$ and set $D= D'-8$. Suppose $\alpha$ is a multicurve and
		$Y$ an essential subsurface with $d_S(\partial Y, \alpha) = D'$. Let $g\in\Mod(Y)$ be fully supported and for each $k\in\Z$, consider the multicurve $\alpha_k= g^k \cdot \alpha$. 
		Then there is a uniform $N$ so that for any choice of infinite, nontrivial subgroups 
		$ H_{kN}\leq \stab(\alpha_{kN})$, 
		the following collection (hence, any subcollection)
		is $D$--separated and $D$--misaligned:
		\[ \{\ldots, H_{-3N}, H_{-2N}, H_{-N}, H_0, H_{N}, H_{2N}, \ldots \} \]

\end{proposition}
\begin{figure}[h]
	        \begin{tikzpicture}
                \def\r{2}
                \def\theta{40}
                \coordinate[label=below:{$\partial Y$}, circle, fill=gray, scale=0.2] (0) at (0,-0.5);
                \coordinate[label=above:{$\partial H_{0}$}, circle, fill=black, scale=0.2] (1) at ({\r*cos(90)}, {\r*sin(90)});

                \coordinate[label=above right:{$\partial H_{N}$}, circle, fill=black, scale=0.2] (2) at ({\r*cos(90-\theta)}, {\r*sin(90-\theta)});
                \coordinate[label=right:{$\partial H_{2N}$}, circle, fill=black, scale=0.2] (3) at ({\r*cos(90-(2*\theta))}, {\r*sin(90-(2*\theta))});
                \coordinate[label=above left:{$\partial H_{-N}$}, circle, fill=black, scale=0.2] (4) at ({\r*cos(90+\theta)}, {\r*sin(90+\theta)});
                \coordinate[label=left:{$\partial H_{-2N}$}, circle, fill=black, scale=0.2] (5) at ({\r*cos(90+(2*\theta))}, {\r*sin(90+(2*\theta))});
                \foreach \i in {1, ..., 5}{
                        \draw[gray] (0) -- (\i);
                }

                \foreach \k in {1, ..., 5}{
                        \foreach \j in {1, ..., 5}{
                                \ifthenelse{\equal{\k}{\j}}{}{\draw (\k) .. controls (0,-.35) .. (\j);
                                }
                        }
                }
                \draw [black,dotted,domain= {90+(2*\theta) +5 }: {90+(2*\theta)+15}] plot ({\r*cos(\x)}, {\r*sin(\x)});
                \draw [black,dotted,domain= {90-(2*\theta) -5 }: {90-(2*\theta)-15}] plot ({\r*cos(\x)}, {\r*sin(\x)});
        \end{tikzpicture}
	\caption{ The convex-hull of the boundaries of the $D$--separating, $D$--misaligned collection  $\{H_{kN}\}_{k\in\Z}$ in $\Cc(S)$. \label{fig: D-sep-misaligned}}
\end{figure}
\begin{proof}
Since $d_S(\alpha,\partial Y)>2$, we have that 
$\pi_Y(g^k \alpha)\neq \emptyset$ for all $k$. 
		 Fixing a component $\alpha_0\in \alpha$, choose $N$ (using  \Cref{thm:minimal_translation_mangahas_cor}) such that for all $|n|\geq N$, we have $d_Y(\alpha_0, g^{n}\alpha_0)>\bgit$. Applying the converse of \Cref{thm:BddGeodesicImage} to $\alpha_0$ and $g^{kN}\alpha_0$ 
	for all $k\in\mathbb Z$, we have that there exists a vertex $\gamma$ on a geodesic $[\alpha_0, g^{kN}\cdot\alpha_0]$ with $d_S(\gamma, \partial Y) \leq 1$. It follows that $D'+1$ is an upper bound for both $d_S(\alpha_0, \gamma), d_S(\gamma, g^{kN}\alpha_0)$ and $D'-2$ is a lower bound for both $d_S(\alpha_0, \gamma)$ and $d_S(\gamma, g^{kN}\cdot\alpha_0)$. So
	\[ 
		2D'-4 \leq d_S(\alpha_0, g^{kN}\cdot \alpha_0) \leq 2D'+2
	\]
	Since $g^{kN}\alpha_0\in R(H_k)$ for all $k$, we have  $d_S(g^{kN}\alpha_0,\partial H_{kN})\leq 1$, hence
	\begin{align*}
	2D'-6 \leq 
	d_S(\alpha_0, g^{(j-i)N}\alpha_0)-2 \leq 
	d_S(\partial H_{iN}, \partial H_{jN}).
	\end{align*}
	It follows that the collection $\{\partial H_{kN}\}_{k\in\Z}$ is $D$--separated, and a straight-forward calculation shows that it is also $D$--misaligned: For $i,j,k$ pairwise distinct, 
	\[(\partial H_{iN}\mid \partial H_{jN})_{\partial H_{kN}} \geq \frac{1}{2}\big((2D'-6)+(2D'-6) -(2D'+4)\big) = D'-8=D.\qedhere\]
\end{proof}

\subsection{The separating assumption for \Cref{main-Loa}}

First we observe that $D$--separating is necessary for \Cref{main-Loa}.
Consider curves $\alpha$, $\beta$ with $d_S(\alpha, \beta)=2$ and large enough geometric intersection number that the group generated by the respective Dehn twists $\langle T_\alpha, T_\beta\rangle$ is a free group. 
The element $T_\alpha T_\beta$ is infinite order and is not conjugate into either factor, yet its orbits in $\Cc(S)$ are bounded.
Thus the coned-off Cayley graph $\hat\Gamma(\langle T_\alpha, T_\beta\rangle, \{ \langle T_\alpha\rangle, \langle T_\beta\rangle\})$ 
does not $\Mod(S)$--equivariantly quasi-isometrically embed into $\Cc(S)$ and $\langle T_\alpha,T_\beta\rangle$ is not RGF.

\subsubsection*{Some setup} We use the following notation in the next two examples, which demonstrate that while the $D$--separating condition is necessary, it alone is insufficient for the conclusions of \Cref{main-Loa}. Fix $D\ge8$ and let $\bgit$ be the constant from \Cref{thm:BddGeodesicImage}. 
Let $ H_{\alpha}$, $H_{\beta}$ and $H_{\gamma} $ denote arbitrary infinite subgroups of the subgroup of 
$\Mod(S)$ generated by the Dehn twists about the components of $\alpha,\beta,$ and $\gamma$ respectively. Recall from \Cref{lem:multitwist} that $H_\alpha, H_\beta,$ and $H_\gamma$ are multitwist groups and are free abelian. 
For a given element $T\in H_{\alpha}$, let 
\[ 
	G \colonequals \langle H_{\alpha}, H_{\beta}, T H_{\gamma} T^{-1}\rangle \quad \textrm{ and }\quad \mathcal{H} \colonequals \{ H_{\alpha}, H_{\beta}, T H_{\gamma} T^{-1} \}.
\]
Observe that $\alpha$ (resp. $\beta$, $T\gamma$) contains $\partial H_\alpha$ (resp. $\partial H_\beta$, $\partial (T H_\gamma T^{-1})$). 
We will use  $\alpha_i,\beta_j,\gamma_k$ to denote arbitrary components of $\alpha,\beta,$ and $\gamma$, respectively, and 
$Y_i$ to denote the annulus about the component $\alpha_i$ of $\alpha$. 

\begin{example}
\label{ex:free_product}
We show that the multicurves $\alpha, \beta, \gamma$ and $T$ may be chosen so that $\mathcal{H}$ is $D$--separated, but does not generate the free product  $H_\alpha\ast H_\beta \ast T H_\gamma T^{-1}$. 
Assume $\gamma=\beta$ and $d_S(\alpha,\beta) \geq D$. 
By raising to sufficiently high powers 
and appealing to
\Cref{thm:minimal_translation_mangahas_cor}, 
 we may assume the element $T\in H_\alpha$ satisfies 
$d_{Y_i}(\beta_j,T\beta_j)\geq \bgit$ for some domain $Y_i$ of $T$, as defined in \Cref{ssub:MCG_Reducibles}. Now, since $\partial Y_i=\alpha_i$ and $T\alpha_i=\alpha_i$,
we can apply the converse of \Cref{thm:BddGeodesicImage} to the geodesic $[\beta_j,T   \beta_j]$
to see that 
\[
  d_S(\partial H_\beta, T \partial H_\beta )) \geq  d_S(\beta_j,T\beta_j) -2 \geq 
  d(\beta_j,\alpha_i)+d(\alpha_i,T\beta_j) - 4 \geq 
  2D-8\geq D.
\]
Then since 
$d_S(\partial H_\alpha, \partial H_\beta) = d_S(\alpha, \beta)\geq D$, 
the collection $\Hc$ is $D$--separated. However, 
$G=\langle H_\alpha,H_\beta\rangle$, and thus is not isomorphic to 
$ H_\alpha  \ast H_{\beta} \ast T H_{\beta} T^{-1}$ as in the conclusions of \Cref{main-Loa}. 
We note also that 
the collection is only at most $1$--misaligned since $(\beta_j\mid T\beta_j)_{\alpha_i}= 1$. See \Cref{fig:nonexample_freeproduct}.  
\end{example}

\begin{figure}[h]
\centering
%
%

\begin{tikzpicture}[scale=1.2,rotate=20]

\draw(170:3) coordinate(beta) --(0,0) coordinate (alpha) node[midway, yshift=1em] {$\geq D$} --(-65:3) coordinate(Tbeta) node[midway,xshift=2em] {$\geq D$} {};

\draw[red] (beta) .. controls (alpha) .. coordinate[pos=.5] (xi) (Tbeta);

\begin{scope}[]
\draw[] (xi)-- coordinate[pos=.5] (mid) (alpha) {};

\draw[<-] (mid)++ (-45:.1) 
.. controls (-45:.3) and (-30:.4) ..
(10:.5) node[above right] {$1$};
\end{scope}

  \def\s{.05}
  \draw[fill] (beta) circle (\s) node[above] {$\beta_j$};
  \draw[fill] (alpha) circle (\s) node[above right] {$\alpha_i$};
   \draw[fill] (Tbeta) circle (\s) node[right] {$T\beta_j$};
   \draw[fill] (xi) circle (\s) node[below left] {$\xi$};
\end{tikzpicture}

\caption{The curves $\alpha_i,\beta_j,T\beta_j$ in \Cref{ex:free_product} are designed so that there exists a $\xi\in[\beta_j,T\beta_j]$ distance $1$ from $\alpha_i$, hence $\mathcal H$ is only 1-misaligned.}
\label{fig:nonexample_freeproduct}
\end{figure}

We note that this does not rule out that
that $G$ is RGF relative to $\mathcal H$---in fact, the group is RGF relative to a different collection, $\{H_\alpha, H_\beta\}$.
However, this conclusion does not hold in general, as demonstrated in the next example.

\begin{example}
\label{E:qi_embedding_fail}
Here we demonstrate that $\alpha, \beta, \gamma$ and $T$ may be chosen so that $\Hc$ is $D$--separated, but $\hat\Gamma(G,\Hc)$ fails to admit a $\Mod(S)$--equivariant quasi-isometric embedding into $\Cc(S)$.

Now assume $\beta$ and $\gamma$ 
intersect but do not fill $S$, and  $\alpha$ satisfies $d_S(\alpha, \beta), d_S(\alpha, \gamma)\geq D$. 
Then there are components $\alpha_i,\beta_j,\gamma_k$ such that $\pi_{Y_i}(\gamma_k)\neq\emptyset$ and 
$\pi_{Y_i}(\beta_j)\neq \emptyset$.
Again using \Cref{thm:minimal_translation_mangahas_cor}, we may assume $T\in H_\alpha$ is such 
that $d_{Y_i}(\beta_j,T\beta_j) \geq \bgit +  d_{Y_i}(\beta_j, \gamma_k)$, thus
\begin{align*}
d_{Y_i}(\beta_j,T\gamma_k)&\geq  
d_{Y_i}(\beta_j,T\beta_j)-d_{Y_i}(T\beta_j,T\gamma_k) 
\\ &= d_{Y_i}(\beta_j,T\beta_j)-d_{Y_i}(\beta_j,\gamma_k) \geq \bgit.
\end{align*} 
We now apply the converse of \Cref{thm:BddGeodesicImage} to conclude
\begin{align*}
d_S(\partial H_\beta, T\partial H_\gamma ) \geq d_S(\beta_j,T\gamma_k) -2 \geq d_S(\beta_j,\alpha_i)+d_S(\alpha_i,\gamma_k)-4 \geq 2D-8\geq D.
\end{align*}
Thus, the collection $\mathcal H$ is $D$--separated. 

Now, for any nontrivial $h\in H_\beta$ and $g\in H_\gamma$, the element $hg\in G$ is an infinite order reducible element. However, $hg=hT^{-1}g' T$, where  $g'\in T H_\gamma T^{-1}$ and $T\in H_\alpha$, is not conjugate into any group in $\mathcal H$ and thus  acts loxodromically on $\hat\Gamma(G,\mathcal H)$. 
Hence 
$\hat\Gamma (G, \mathcal{H})$ admits no $\Mod(S)$--equivariant quasi-isometric embedding into the curve graph, and $G$ is not RGF relative to $\mathcal H$.
Observe again (see \Cref{fig:nonexample_qiembedding}) that the collection is only 1-misaligned.
\end{example}

\begin{figure}[h]
\centering
%
%
%

\begin{tikzpicture}[scale=1.2,rotate=10]

\draw(160:3) coordinate(gamma) --(0,0) coordinate (alpha) node[midway, yshift=1em] {$\geq D$} --(-65:3) coordinate(Tbeta) node[midway,xshift=2em] {$\geq D$} {};
\draw (170:3) coordinate (beta) -- (alpha) -- (-75:3) coordinate (Tgamma) ;

\draw[red] (beta) .. controls (alpha) .. coordinate[pos=.5] (xi) (Tgamma);

\begin{scope}[]
\draw[] (xi)-- coordinate[pos=.5] (mid) (alpha) {};

\draw[<-] (mid)++ (-45:.1) 
.. controls (-45:.3) and (-30:.4) ..
(10:.5) node[above right] {$1$};

\draw (gamma)++(-.2,.1) coordinate (sgamma);
\draw (beta)++(-.2,0) coordinate (sbeta);

\draw (Tgamma)++(0,-.2) coordinate (sTgamma);
\draw (Tbeta)++(.1,-.2) coordinate (sTbeta);

\draw [decorate,decoration={brace,amplitude=3pt}]
(sbeta)--(sgamma)
node [midway,xshift=-10pt,yshift=2pt] {$2$};

\draw [decorate,decoration={brace,amplitude=3pt}]
(sTbeta)--(sTgamma)
node [midway,xshift=3pt,yshift=-10pt] {$2$};
\end{scope}

  \def\s{.05}
  \draw[fill] (beta) circle (\s) node[below] {$\beta_j$};
  \draw[fill] (alpha) circle (\s) node[above right] {$\alpha_i$};
   \draw[fill] (Tbeta) circle (\s) node[right] {$T\beta_j$};
  \draw[fill] (gamma) circle (\s) node[above] {$\gamma_k$};
   \draw[fill] (Tgamma) circle (\s) node[left] {$T\gamma_k$};
   \draw[fill] (xi) circle (\s) node[below left] {$\xi$};
\end{tikzpicture}

\caption{The curves $\alpha_i,\beta_j,T\gamma_k$ in \Cref{E:qi_embedding_fail} are designed so that there exists a $\xi\in[\beta_j,T\gamma_k]$ distance $1$ from $\alpha_i$, hence $\mathcal H$ is only 1-misaligned.}
\label{fig:nonexample_qiembedding}
\end{figure}

	\subsection{The torsion-free assumption for \Cref{main-Loa}}
	Here we'll demonstrate the necessity of the the torsion-free assumption in \Cref{main-Loa} by producing a $D$--separated collection $\{H_1, H_2\}$ with $H_1$, $H_2$ containing torsion, but the group $\langle H_1, H_2 \rangle$ does not split as the free product of the factors.

	\begin{example}
	\label{E:example_torsion_free}
	Assume there exists an element $\sigma \in \Mod(S)$ of order $k$ which fixes a multicurve $\alpha$, and 
	that $f$ is a pseudo-Anosov which commutes with $\sigma$. For instance, $\sigma$ could represent an order--$2$ homemorphism fixing a multicurve $\alpha$ such that the quotient $S/\langle \sigma\rangle$ is a $2$--orbifold which admits a pseudo-Anosov element $\bar f$. The homeomorphism $\bar f$ lifts to a pseudo-Anosov $f$ on $S$ which commutes with $\sigma$.

	Now, let $T=T_\alpha$ denote the composition of Dehn twists along each component of $\alpha$ and observe that $\sigma$ commutes with $T$ .  Let $H_1=\langle \sigma, T\rangle$ and $H_2=f^\ell H_1 f^{-\ell}$ for $\ell$ large. Note that $H_1$ and $H_2$ are both reducible and isomorphic to $\mathbb Z\times \mathbb Z_k$. 
 See that $\partial H_1=\alpha$ and $\partial H_2=f^\ell\alpha$, so for any $D>0$ we can choose $\ell$ large enough so that $\{H_1,H_2\}$ is $D$--separated. However, $\sigma\in H_2$ as well since it commutes with $f$, hence $\langle H_1, H_2 \rangle \cong (\Z \ast \Z)\times \Z_k$ $\not\cong 
	H_1 \ast H_2$.

 We show that $G=\<H_1, H_2\>$ is nonetheless RGF relative to $\{H_1,H_2\}$.	It is straightforward that $G$ satisfies the bounded coset penetration property with respect to the subgroups $H_1,H_2$, hence $G$ is hyperbolic relative to $\{H_1,H_2\}$.
	
	Let $H_1'=\langle T \rangle$ and $H_2'=\langle f^\ell T f^{-\ell}\rangle$.
	First, see that 
	\[
	G =\langle f^\ell,T,\sigma \mid [f^\ell,\sigma], [T,\sigma], \sigma^k\rangle\cong (\mathbb Z\ast \mathbb Z)\times \mathbb Z_k
	\] contains $G' =\langle H_1',H_2'\rangle\cong\mathbb Z\ast\mathbb Z$ as a finite-index torsion-free subgroup, hence $G$ and $G'$ are quasi-isometric. By \Cref{main-Loa}, $G'$ is RGF relative to $\{H_1',H_2'\}$. 
	We may compose the induced quasi-isometry between $\hat\Gamma (G,\{H_1,H_2\})$ and $\hat\Gamma (G',\{H_1',H_2'\})$, with the quasi-isometric embedding  $\hat\Gamma (G',\{H'_1,H'_2\})\to \Cc(S)$ to yield the $\Mod(S)$-equivariant quasi-isometric embedding in the definition of RGF. 
	\end{example}

	\bibliographystyle{alpha}
	\bibliography{GF-RAAGs}

\end{document}